\documentclass{birkjour}
\usepackage[pdfmenubar=true,pdffitwindow=false,pdfstartview={FitH},colorlinks=true,linkcolor=OliveGreen,citecolor=blue,filecolor=black,urlcolor=red]{hyperref}

\usepackage{graphics,graphicx,float,latexsym,color}
\usepackage[T1]{fontenc}    % use 8-bit T1 fonts

\usepackage{amsthm,amsbsy,amsmath,amssymb,amscd,amsfonts}
\usepackage{booktabs}       % professional-quality tables
\usepackage{nicefrac}       % compact symbols for 1/2, etc.
\usepackage{microtype}      % microtypography
\usepackage{epigraph}
\usepackage[font={small,it}]{caption}
\usepackage{subcaption}

\usepackage{makecell}

\usepackage[dvipsnames]{xcolor}

\newtheorem{theorem}{Theorem}

\newtheorem{proposition}{Proposition}

\newtheorem{corollary}{Corollary}

\newtheorem{lemma}{Lemma}
\newtheorem*{lemma*}{Lemma}

\usepackage{lineno}

\begin{document}

\title[The Ballet of the Triangle Centers on the Elliptic Billiard]{The Ballet of Triangle Centers\\on the Elliptic Billiard}

\author[Reznik]{Dan Reznik$^*$}
\address{Data Science Consulting\\
Rio de Janeiro, RJ, Brazil}
\email{dreznik@gmail.com}
\thanks{* \texttt{dreznik@gmail.com}. 1st author thanks IMPA for the opportunity to present this research at the 2019 Math Colloquium. 2nd author thanks CNPq for a Fellowship and Project PRONEX/CNPq/FAPEG 2017 10 26 7000 508. 3rd author thanks Federal University of Juiz de Fora for a 2018-2019 fellowship.}

\author[Garcia]{Ronaldo Garcia}
\address{Inst. de Matemática e Estatística\\
Univ. Federal de Goiás\\
Goiânia, GO, Brazil}
\email{ragarcia@ufg.br}

\author[Koiller]{Jair Koiller}
\address{Dept. de Matemática\\
Univ. Federal de Juiz de Fora\\
Juiz de Fora, MG, Brazil}
\email{jairkoiller@gmail.com}



\date{}

\dedicatory{To Clark Kimberling, Peter Moses, and Eric Weisstein}

\begin{abstract}
The dynamic geometry of the family of 3-periodics in the Elliptic Billiard is mystifying. Besides conserving perimeter and the ratio of inradius-to-circumradius, it has a stationary point. Furthermore, its triangle centers sweep out mesmerizing loci including ellipses, quartics, circles, and a slew of other more complex curves. Here we explore a bevy of new phenomena relating to (i) the shape of 3-periodics and (ii) the kinematics of certain Triangle Centers  constrained to the Billiard boundary, specifically the non-monotonic motion some can display with respect to 3-periodics. Hypnotizing is the {\em joint} motion of two such non-monotonic Centers, whose many stops-and-gos are akin to a Ballet.
\end{abstract}

\maketitle

\section{Introduction}
\label{sec:intro}
Being uniquely integrable \cite{kaloshin2018}, the Elliptic Billiard (EB) is the {\em avis rara} of planar Billiards. As a special case of Poncelet's Porism \cite{dragovic88}, it is associated with a 1d family of $N$-periodics tangent to a confocal Caustic and of constant perimeter \cite{sergei91}. Its plethora of mystifying properties has been extensively studied, see \cite{rozikov2018-billiards} for a recent treatment.

Through the prism of classic triangle geometry, we initially explored the loci of their {\em Triangle Centers} (TCs) \cite{reznik2019-loci-gallery-orbit-exc}: e.g., the Incenter $X_1$, Barycenter $X_2$, Circumcenter $X_3$, etc., see summary below. The $X_i$ notation is after  Kimberling's Encyclopedia \cite{etc}, where thousands of TCs are catalogued. Appendix~\ref{app:trilinears} reviews basic TCs. Here we explore a bevy of curious behaviors displayed by the the family of 3-periodics, roughly divided into two categories (with an intermezzo):

\begin{enumerate}
    \item The {\em shape} of 3-periodics and/or TC loci: when are the former acute, obtuse, pythagorean, and the latter non-smooth, self-intersecting, non-compact? See for example this recent treatment of TCs at infinity  \cite{kimberling2020-poly-infinity}.
    \item The {\em kinematics} of EB-railed TCs: a handful of TCs\footnote{Some 50 out of 40 thousand in \cite{etc}.} is known to lie on the EB. As the family of 3-periodics is traversed, what is the nature of their motion (monotonicity, critical points, etc.). We further examine the {\em joint} motion of two EB-railed TCs which in some cases resembles a Ballet.
\end{enumerate}

In the intermezzo we introduce (i) a deceptively simple construction under which the EB bugles out the Golden Ratio, and (ii) a triangle closely related to 3-periodics\footnote{The Contact (or Intouch) Triangle of the Anticomplementary Triangle (ACT).}, whose vertices are dynamically clutched onto the EB.

Throughout the paper Figures will sometimes include a hyperlink to an animation in the format \cite[PL\#nn]{reznik2020-playlist-intriguing}, where nn is the entry into our video list on Table~\ref{tab:playlist} in Section~\ref{sec:conclusion}. Indeed, the beauty of most phenomena remain ellusive unless they are observed dynamically.

\paragraph{Recent Work} Two particularly striking observations have been \cite{reznik2020-intelligencer}:

\begin{enumerate}
    \item The Mittenpunkt $X_9$ is the black swan of all TCs: its locus is a {\em point} at the center of the EB\footnote{In Triangle parlance, the EB is the ``$X_9$-Centered Circumellipse''.} \cite[PL\#01]{reznik2020-playlist-intriguing}.
    \item The 3-periodic family conserves the ratio of Inradius-to-Circumradius. This implies the sum of its cosines is invariant. Suprisingly the latter holds for all $N$-periodics \cite{akopyan2020-invariants,bialy2020-invariants}.
\end{enumerate}

\noindent Other observations focused on the surprising elliptic locus of many a TC: the locus of the Incenter $X_1$ is an ellipse \cite{olga14}, as is that of the Excenters \cite{garcia2019-incenter}. The latter is similar to a rotated locus of $X_1$, Figure~\ref{fig:incenter-loci}. Also elliptic are the loci of the Barycenter $X_2$ \cite{sergei2016proj}, Circumcenter $X_3$ \cite{corentin19}, Orthocenter $X_4$ \cite{garcia2019-incenter}, Center $X_5$ of the Nine-Point Circle  \cite{garcia2020-ellipses}. See Figure~\ref{fig:x12345-feuer-combo} (left). Recently we showed that out of the first 100 Kimberling Centers in \cite{etc}, exactly 29 produce elliptic loci \cite{garcia2020-ellipses}. This is quite unexpected given the non-linearities involved.

Other observations, many which find parallels in Triangle Geometry, included (i) the locus of the Feuerbach Point $X_{11}$ is on the Caustic\footnote{The Inconic centered on the Mittenpunkt $X_9$ which passes through the Extouchpoints is known as the {\em Mandart Inellipse} \cite{mw}. By definition, an Inconic is internally tangent to the sides, so it must be the Caustic.}, as is that of the Extouchpoints, though these move in opposite directions; (ii) the locus of $X_{100}$, the anticomplement of $X_{11}$, is on the EB \cite{garcia2020-new-properties}. See Figure~\ref{fig:x12345-feuer-combo} (right); (iii) the locus of vertices of Intouch, Medial and Feuerbach Triangles are all non-elliptic. See Figure~\ref{fig:non-elliptic-vertex}.

We still don't understand how loci ellipticiy or many of the phenomena below correlate to the Trilinears of a given TC.

\begin{figure}[b]
    \centering
    \includegraphics[width=\textwidth]{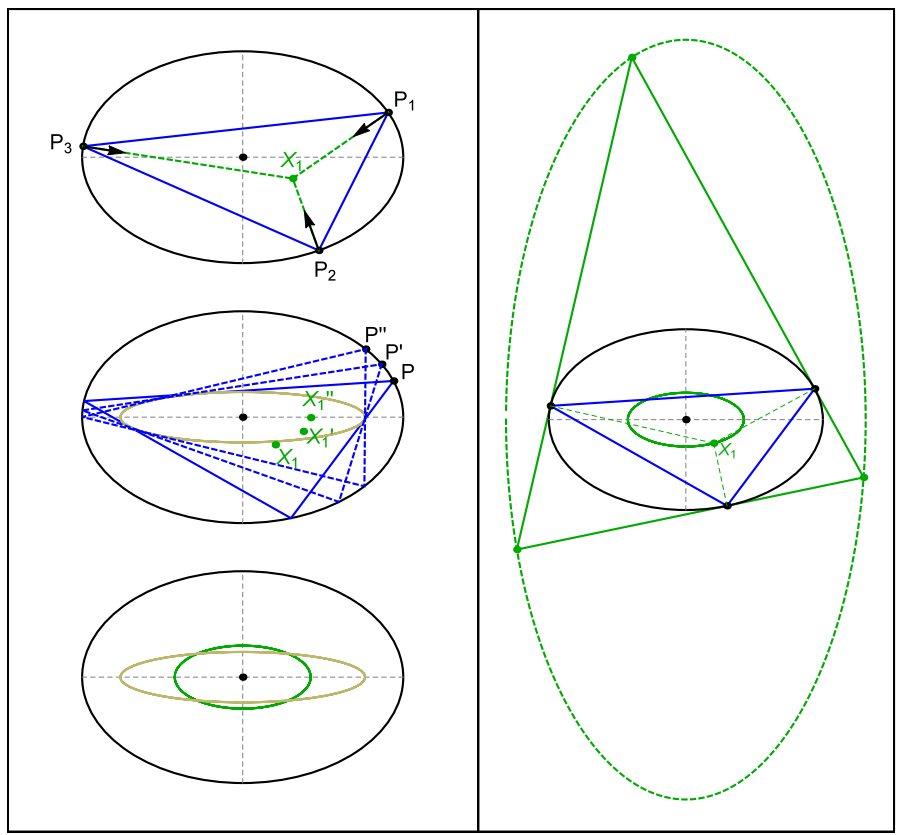}
    \caption{\textbf{Left Top}: A 3-periodic and its Incenter $X_1$: where bisectors meet. \textbf{Left Mid}: Three 3-periodics, each identified by a starting vertex $P,P',P''$, and their Incenters $X_1,X_1',X_1''$. Also shown is the  confocal Caustic (brown) which is the stationary {\em Mandart Inellipse} \cite{mw} of the 3-periodic family. \textbf{Left Bot}: the locus of $X_1$ over the 3-periodic family is an ellipse (green). Also shown is the Caustic (brown). \textbf{Right}: the Excentral Triangle (green) of a 3-periodic (blue). The locus of its vertices (the Excenters) is an ellipse (dashed green) similar to a perpendicular copy of the locus of the Incenter (solid green inside the EB). This is the stationary {\em MacBeath} Circumellipse of the Excentral Triangle \cite{mw}, centered on the latter's $X_6$ (i.e., $X_9$) \textbf{Video}:\cite[PL\#01,04]{reznik2020-playlist-intriguing}.}
    \label{fig:incenter-loci}
\end{figure}

\begin{figure}
\centering
\includegraphics[width=\linewidth]{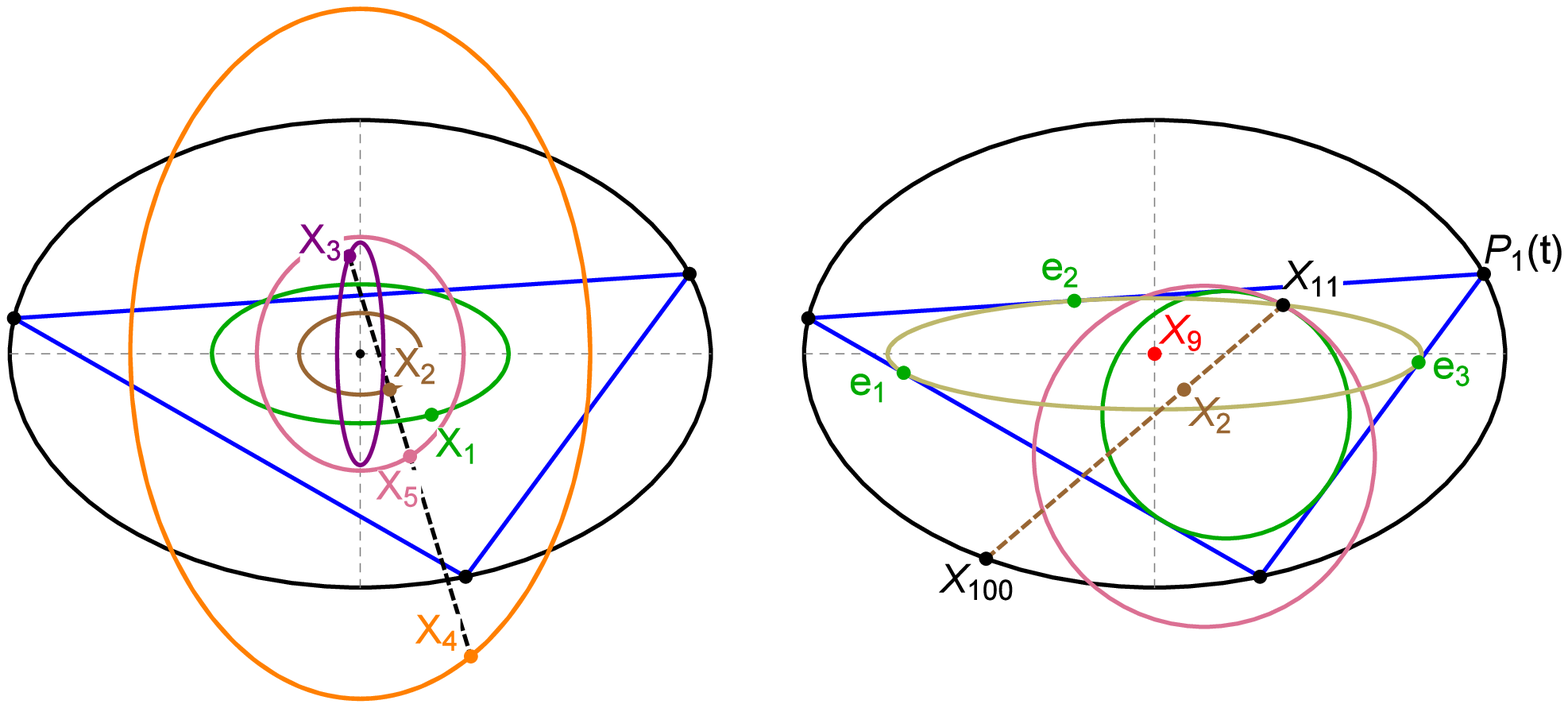}
\caption{\textbf{Left}: The loci of Incenter $X_1$ (green), Barycenter $X_2$ (brown), Circumcenter $X_3$ (purple), Orthocenter $X_4$ (orange), and Center of the 9-Point Circle $X_5$ (pink), are all ellipses, \href{https://youtu.be/sMcNzcYaqtg}{Video} \cite[PL\#05]{reznik2020-playlist-intriguing}. Also shown is the {\em Euler Line} (dashed black) which for any triangle, passes through all of $X_i$, $i=1...5$ \cite{mw}. \textbf{Right}: The locus of $X_{11}$, where the Incircle (green) and 9-Point Circle (pink) meet, is the Caustic (brown), also swept (in the opposite direction) by the Extouchpoints $e_i$. $X_{100}$ (double-length reflection of $X_{11}$ about $X_2$) sweeps the EB. $X_9$ is the black swan of all points: it is stationary at the EB center \cite{reznik2020-intelligencer}. \textbf{Video:} \cite[PL\#05,07]{reznik2020-playlist-intriguing}.}
\label{fig:x12345-feuer-combo}
\end{figure}

\begin{figure}
    \centering
    \includegraphics[width=\linewidth]{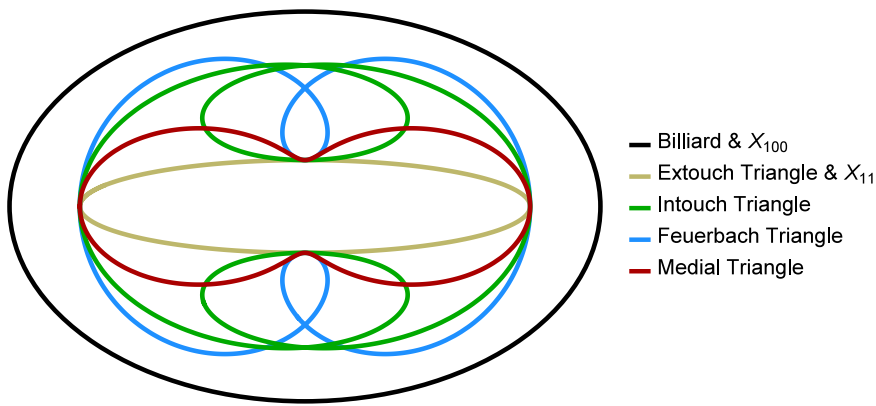}
    \caption{The loci of the Intouch (green), Feuerbach (blue), and Medial (red) Triangles which are all non-elliptic. An exception is the locus of the Extouchpoints (brown), who sweep the elliptic Caustic (as does $X_{11}$ not shown). \cite[PL\#07]{reznik2020-playlist-intriguing}}
    \label{fig:non-elliptic-vertex}
\end{figure}

\clearpage
\section{First Movement: Shape}
\label{sec:act-I}
Let the boundary of the EB be given by, $a>b>0$:

\begin{equation}
f(x,y)=\left(\frac{x}{a}\right)^2+\left(\frac{y}{b}\right)^2=1
\label{eqn:billiard-f}
\end{equation}

\noindent Below $c^2=a^2-b^2$, and $\delta=\sqrt{a^4-a^2 b^2+b^4}$.

Throughout this paper we assume one vertex $P_1(t)=(x_1,y_1)$ of a 3-periodic is parametrized as $P_1(t)=[a\cos{t},b\sin{t}]$. Explicit expressions\footnote{Their coordinates involve double square roots on $x_1$ so these points are constructible by ruler and compass.} for the locations of $P_2(t)$ and $P_3(t)$ under this specific parametrization are given in \cite{garcia2019-incenter}.

\subsection{Can 3-periodics be obtuse? Pythagorean?}
\label{sec:x4}
The locus of the Orthocenter $X_4$ is an ellipse of axes $a_4,b_4$ similar to a rotated copy of the EB. These are given by \cite {garcia2020-ellipses}:

\begin{equation*}
\left(a_4,b_4\right)=\left(\frac{k_4}{a},\frac{k_4}{b}\right),\;  k_4=\frac{  ({a}^{2}+{b}^{2})\delta-2\,{a}^{2}{b}^{2} }{c^2}    
\end{equation*}

\noindent Referring to Figure~\ref{fig:orthocenter_loci}, let $\alpha_4=\sqrt{2\,\sqrt {2}-1}\;{\simeq}\;1.352$.

\begin{proposition}
If $a/b=\alpha_4$, then $b_4=b$, i.e., the top and bottom vertices of the locus of $X_4$ coincide with the EB's top and bottom vertices.
\end{proposition}

\begin{proof}
The equation $b_4=b$ is equivalent to $a^4+2a^2b^2-7b^4=0.$ Therefore, as $a>b>0$, it follows that $a/b=\sqrt{2\,\sqrt {2}-1}.$
\end{proof}

\noindent Let $\alpha_4^*$ be the positive root of
${x}^{6}+{x}^{4}-4\,{x}^{3}-{x}^{2}-1=0$, i.e.,
$\alpha_4^{*}={\simeq}\;1.51$. 

\begin{proposition}
When $a/b=\alpha_4^{*}$, then $a_4=b$ and $b_4=a$, i.e., the locus of $X_4$ is identical to a rotated copy of Billiard. 
\end{proposition}

\begin{proof}
The condition $a_4=b$, or equivalently $b_4=a$, is defined by $a^6+a^4b^2-4a^3b^3-a^2b^4-b^6=0$. Graphic analysis shows that ${x}^{6}+{x}^{4}-4\,{x}^{3}-{x}^{2}-1=0$ has only one positive real root which we call $\alpha_4^*$.
\end{proof}

\begin{theorem}
If $a/b<\alpha_4$ (resp. $a/b>\alpha_4$) the 3-periodic family will not (resp. will) contain obtuse triangles.
\end{theorem}

\begin{proof}
If the 3-periodic is acute, $X_4$ is in its interior, therefore also internal to the EB. If the 3-periodic is a right triangle, $X_4$ lies on the right-angle vertex and is therefore on the EB. If the 3-periodic is obtuse, $X_4$ lies on exterior wedge between sides incident on the obtuse vertex (feet of altitudes are exterior). Since the latter is on the EB, $X_4$ is exterior to the EB.
\end{proof}

\begin{figure}[H]
    \centering
    \includegraphics[width=\textwidth]{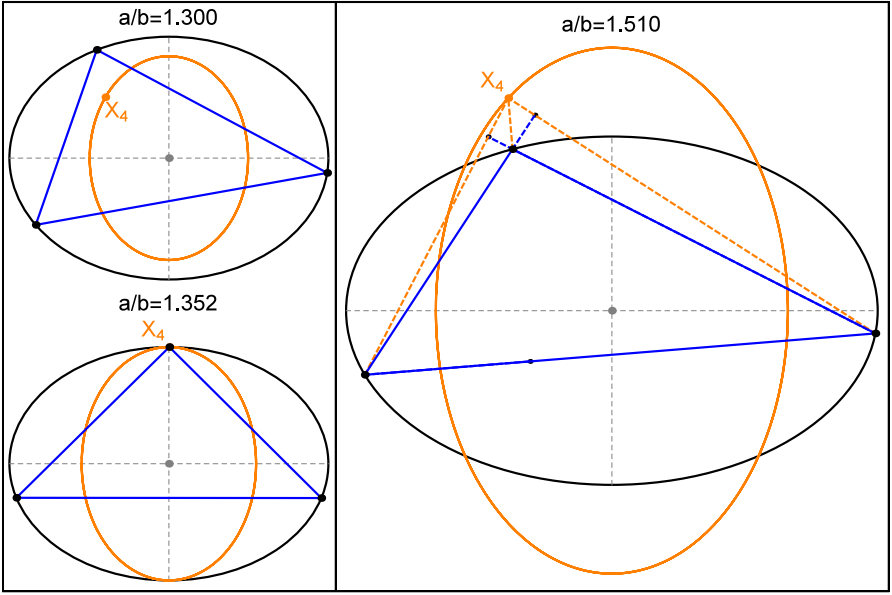}
    \caption{Let
    $\alpha_4=\sqrt{2\,\sqrt {2}-1}\;{\simeq}\;1.352$
    and $H$ be the elliptic locus of $X_4$ (orange), similar to a rotated copy of the EB (black). \textbf{Top Left}: $a/b<\alpha_4$, $H$ is interior to the EB and all 3-periodics (blue) are acute. \textbf{Bot Left}: at $a/b=\alpha_4$, $H$ is tangent to the top and bottom vertices of the EB. The 3-periodic shown is a right triangle since one vertex is at the upper EB vertex where $X_4$ currently is. \textbf{Right}: at $a/b>\alpha_4$, the 3-periodic family will contain both acute and obtuse 3-periodics, corresponding to $X_4$ interior (resp. exterior) to the EB. For any obtuse 3-periodic, $X_4$ will lie within the wedge between sides incident upon the obtuse angle and exterior to the 3-periodic, i.e., exterior to the EB. For the particular aspect ratio shown ($a/b=1.51$), $H$ is identical to a $90^{\circ}$-rotated copy of the EB.}
    \label{fig:orthocenter_loci}
\end{figure}

\subsection{Can a locus be non-smooth?}
\label{sec:orthic-incenter}
Loci considered thus far have been smooth, regular curves.

\begin{proposition}
If $a/b>\alpha_4$ the locus of the Incenter of the 3-periodic's Orthic Triangle comprises four arcs of ellipses, connected at four corners.
\end{proposition}

To see this, let $T$ be a triangle, $T_h$ its Orthic\footnote{Its vertices are the feet of the altitudes.}, and $I_h$ be the latter's Incenter. It is well-known that if $T$ is acute $I_h$ coincides with $T$'s Orthocenter $X_4$. However, for obtuse $T$:

\begin{lemma*}
$T_h$ has to vertices outside of $T$, and $I_h$ is ``pinned'' to the obtuse vertex \cite[Chapter 1]{coxeter67}.
\end{lemma*}

This is restated in Appendix~\ref{app:orthic-incenter} as Lemma~\ref{lem:pinned} followed by a proof. This curious phenomenon is illustrated in Figure~\ref{fig:orthic-incenter}. 

\begin{figure}[H]
    \centering
    \includegraphics[width=.75\textwidth]{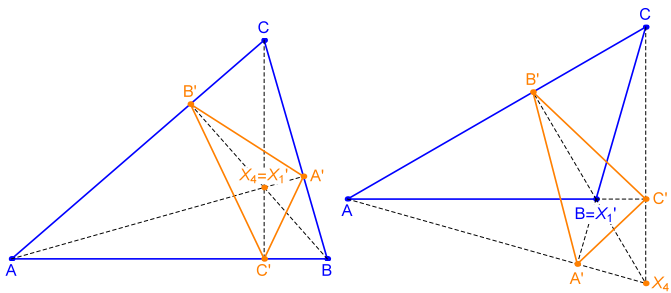}
    \caption{\textbf{Left}: If $T=ABC$ is acute (blue), its Orthic $T'=A'B'C'$ is the so-called Fagnano Triangle \cite{rozikov2018-billiards}, whose properties include: (i) inscribed triangle of minimum perimeter, (ii) a 3-periodic of $T$, i.e., the altitudes of $T$ are bisectors of $T'$, i.e., the Orthic Incenter $X_1'$ coincides with the Orthocenter $X_4$. \textbf{Right}: If $T$ is obtuse, two of the Orthic's vertices lie outside $T$, and $X_4$ is exterior to $T$. $T'$ is the Orthic of {\em both} $T$ and {\em acute} triangle $T_e=A{X_4}C$. So the Orthic is the latter's Fagnano Triangle, i.e., $B$ is where both altitudes and bisectors meet. The result is that if $T$ is obtuse at $B$, the Incenter $X_1'$ of the Orthic is $B$. \textbf{Video}: \cite[PL\#06]{reznik2020-playlist-intriguing}}
    \label{fig:orthic-incenter}
\end{figure}

Assume $a/b>\alpha_4$. Since the 3-periodic family contains both acute and obtuse triangles, the locus of $I_h$ transition between acute and obtuse regimes:

\begin{center}
\begin{tabular}{r|c|l}
 3-periodic & $X_4$ wrt EB & $I_h$ location \\ 
 \hline
 acute & interior & Orthocenter $X_4$ \\  
 right triangle & on it & right-angle vertex \\
obtuse & external (3-periodic Excenter) & obtuse vertex, on EB 
\end{tabular}
\end{center}

In turn, this yields a locus for $I_h$ consisting of four elliptic arcs connected by their endpoints in four corners, Figure~\ref{fig:orthic_incenter_locus}. Notice top and bottom (resp. left and right) arcs belong to the EB (resp. $X_4$ locus).

\begin{figure}
    \centering
    \includegraphics[width=\textwidth]{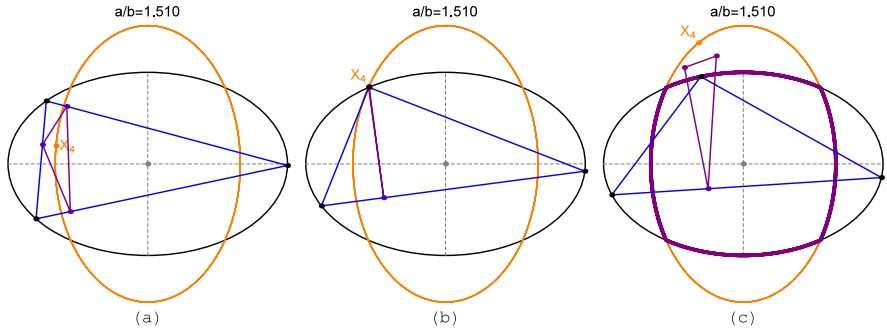}
    \caption{An $a/b>\alpha_4$ EB is shown (black). Let $T$ and $T_h$ be the 3-periodic and its Orthic Triangle (blue and purple, respectively). \textbf{(a)} $T$ is acute ($X_4$ is interior to the EB), and $I_h=X_4$. \textbf{(b)} $X_4$ is on the EB and $T$ is a right triangle. $T_h$ degenerates to a segment. \textbf{(c)} $X_4$ is exterior to the EB. Two of $T_h$'s vertices are outside $T$. $I_h$ is pinned to $T$'s obtuse vertex, on the EB. $X_4$ is an Excenter of the 3-periodic. The complete locus of $I_h$ comprises therefore 4 elliptic arcs (thick purple). \textbf{Video}: \cite[PL\#07]{reznik2020-playlist-intriguing}}
    \label{fig:orthic_incenter_locus}
\end{figure}

For the next proposition, 
let $\alpha_h^2$ (resp.~$1/\alpha_h'^2$) be the real root above 1 (resp.~less than 1) of the polynomial $1 + 12 x - 122 x^2 - 52 x^3 + 97 x^4$. Numerically, $\alpha_h{\simeq}1.174$ and $\alpha_h'{\simeq}2.605$.

%\textcolor{red}{ronaldo da pra trocar o $a_1$ %e $b_1$ (usados para os eixos do locus do %incentro) abaixo por talvez u,v ou outra %notacao?}

\begin{proposition}
At $a/b=\alpha_h$ (resp.~$a/b=\alpha_h'$), at the sideways (resp.~upright) 3-periodic, the Orthic is a right triangle, Figure~\ref{fig:right-orthic}. If $a/b>\alpha_h$ some Orthics are obtuse (a family always contains acutes).
\end{proposition}

\begin{proof}
The orthic of an isosceles triangle with vertices $A=(a,0),$ $B=(-u,v)$ and $C=(-u,-v)$ is the isosceles triangle with vertices:
\begin{align*}
A'=&(-u,0)\\
B'=&\left(\frac{-u(a+u)^2+v^2(2a+u) }{(a+u)^2+v^2},\frac{ v( (a+u)^2-v^2)}{(a+u)^2+v^2)}\right)\\
C'=&\left(\frac{-u(a+u)^2+v^2(2a+u) }{(a+u)^2+v^2},- \frac{v( (a+u)^2-v^2)}{(a+u)^2+v^2}\right)
\end{align*}

It is rectangle, if and only if, $\langle B'-A',C' -A'\rangle=0$. This condition is expressed by
$r(a,u,v)=(a+u)^2-v(2a+2u+v)=0.$

Using explicit expression to the 3-periodic vertices \cite{garcia2019-incenter}, obtain that  
$u= u(a,b)=a (\delta- {b}^{2})/({a}^{2}-{b}^{2})$ and $v=v(a,b)={b}^{2}\sqrt {2\delta-{a}^{2}-{b}^{2} 
}/({a}^{2}-{b}^{2}).
$
So it follows that $r(a,u,\simeq)=r(a,b)=97a^8-52a^6b^2-122a^4b^4+12a^2b^6+b^8=0.$
The same argument can be applied to the isosceles 3-periodic with vertices:

\begin{equation*}
    A=(0,b),\;\;B=(-v(b,a),-u(b,a)),\;\;C=(v(b,a),-u(b,a))
\end{equation*}

\noindent The associated orthic triangle will be rectangle, if and only if, $r(b,a)=0$.
\end{proof}

\begin{figure}
    \centering
    \includegraphics[width=\textwidth]{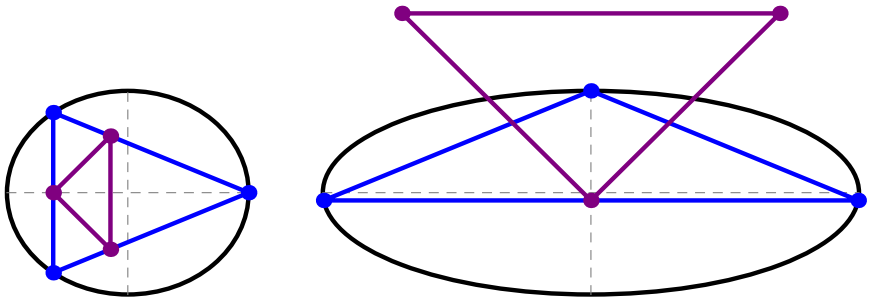}
    \caption{\textbf{Left}: At $a/b=\alpha_h{\simeq}1.174$, a sideways 3-periodic (blue) has a right-triangle Orthic (purple). If $a/b>\alpha_h$ some Orthics in the family are obtuse. \textbf{Right}: At $a/b=\alpha_h'{\simeq}2.605$, when the 3-periodic is an upright isosceles (obtuse since $a/b>\alpha_h$), its extraverted Orthic is also a right triangle.}
    \label{fig:right-orthic}
\end{figure}

The obtuseness of 3-periodic Orthics is a complex phenomenon with several regimes depending on $a/b$. Figure~\ref{fig:orthic-orthic} provides experimental details. 

\begin{figure}
    \centering
    \includegraphics[width=\linewidth]{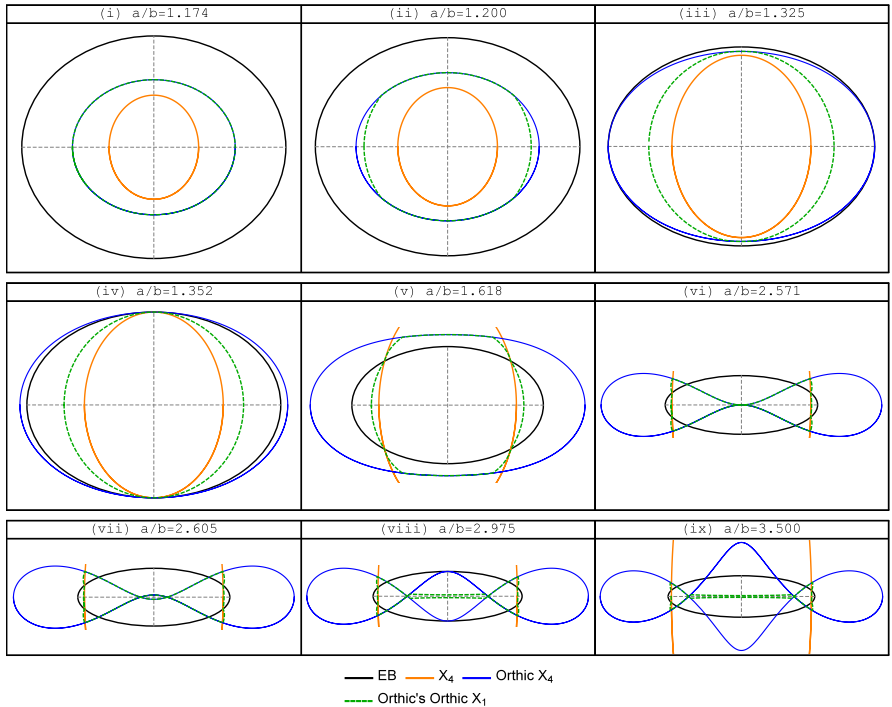}
    \caption{When are Orthics Obtuse? As before, this requires the Orthic Orthic's Incenter $X_1''$ (dashed green) to be detached from the Orthic's Orthocenter $X_4'$ (blue). Let $V$ (resp. $H$) denote either one of the two upright (resp. sideways) isosceles 3-periodics. Notable transitions occur at: (i) $a/b=\alpha_h{\simeq}1.174$, the locus of $X_1''$ is identical to that of $X_4'$, and at $H$, its Orthic is a right-triangle, Figure~\ref{fig:right-orthic} (left); (ii) $a/b>\alpha_h$, one can see $X_1''$ detaching from $X_4'$ indicating a region of obtuse Orthics about the $H$; (iii) At $a/b{\simeq}1.325$ the locus of $X_4'$ touches the EB's left and right vertices at the $H$; (iv) At $a/b=\alpha_4{\simeq}1.352$, the loci of $X_4$ of $X_4'$, and $X_1''$ touch the EB's top and bottom vertices, indicating $V$ are right-triangles and all Orthics not stemming from these are obtuse; (v) At $a/b>\alpha_4$, $X_1''$ tracks $X_4'$ about $V$, indicating some Orthics are acute; (vi) At $a/b{\simeq}2.571$, $X_4'$ two acute Orthics pass through the origin. Above this threshold, the locus of $X_4'$ becomes self-intersecting on the horizontal axis of the EB; (vii) At $a/b=\alpha_h'{\simeq}2.605$, $V$ yield right-triangle Orthics, Figure~\ref{fig:right-orthic} (right). Just above this threshold a new region of obtuse Orthics flares up about $V$; (viii) at $a/b{\simeq}2.965$ $X_4'$ of two obtuse Orthics touch top and bottom vertex of the EB at $V$; (ix) as $a/b$ increases, Orthics about $V$ become more and more obtuse (judging from the deviation between blue and dashed green curves), however two sideway regions of acute Orthic remain where $X_1''$ tracks $X_4'$: these are squeezed to the left and right of the self-intersections of $X_4'$ and the locus of $X_4$. \textbf{Video}: \cite[PL\#08]{reznik2020-playlist-intriguing}}
    \label{fig:orthic-orthic}
\end{figure}

\clearpage
\subsection{Can a Locus be Self-Intersecting?}
\label{sec:x59}
\noindent
\epigraph{The trees are in their autumn beauty,\\
The woodland paths are dry,\\
Under the October twilight the water\\
Mirrors a still sky;\\
Upon the brimming water among the stones\\
Are nine-and-fifty swans.}{\textit{W.B. Yeats}}

Yeats points us an elegant TC: $X_{59}$, the ``Isogonal Conjugate of $X_{11}$'' \cite{etc}, i.e., the two centers have reciprocal trilinears.

Experimentally, $X_{59}$ is a continuous curve internally tangent to the EB at its four vertices, and with four self-intersections, Figure~\ref{fig:x59-center}, and as an animation \cite[PL\#12]{reznik2020-playlist-intriguing}. It intersects a line parallel to and infinitesimally away from either axis on six points, so its degree must be at least 6. Producing analytic expressions for salient aspects of its geometry has proven a tough nut to crack, namely, the following are unsolved:

\begin{itemize}
\item What is the degree of its implicit equation?
\item What is $t$ in $P_1(t)=\left(a\cos(t),b\sin(t)\right)$ such that $X_{59}$ is on one of the four self-intersections? For example, at $a/b=1.3$ (resp. $1.5$), $t$, given in degrees is ${\simeq}32.52^\circ$ (resp. $29.09^\circ$), Figure~\ref{fig:x59-center} (left-bottom).
\item What is $a/b$ such that if $X_{59}$ is on one of the lower self-intersection on the $y$-axis, the 3-periodic is a right triangle? Numerically, this occurs when $a/b=\alpha_{59}^\perp\,{\simeq}\,1.58$ and $t{\simeq}27.92^\circ$, Figure~\ref{fig:x59-center} (right).
\end{itemize}

\begin{figure}
    \centering
    \includegraphics[width=\textwidth]{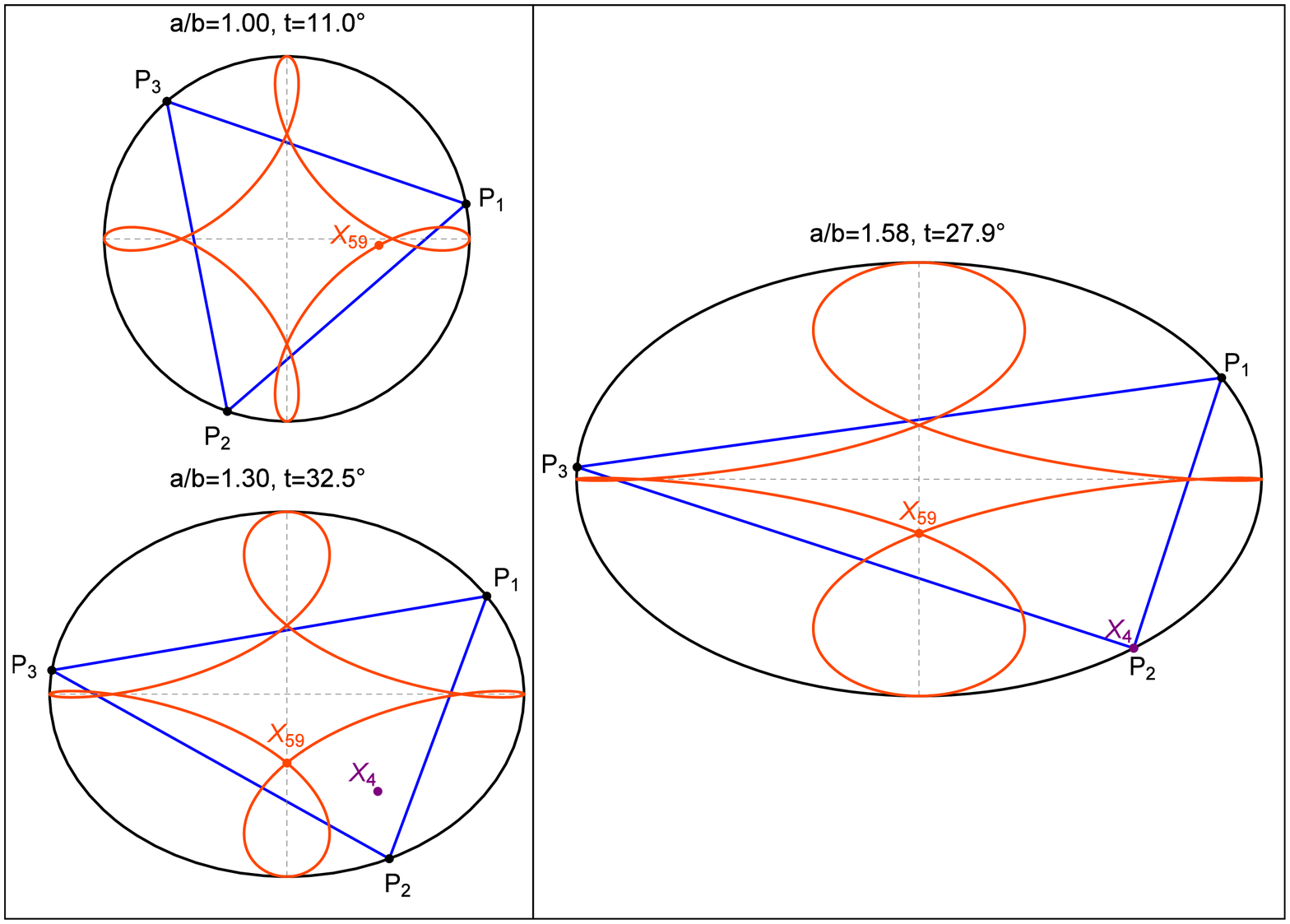}
    \caption{The locus of $X_{59}$ is a continuous curve with four self-intersections, and at least a sextic. It is tangent to the EB at its four vertices. \textbf{Top Left}: circular EB, $X_{59}$ is symmetric about either axis. \textbf{Bottom Left}: $a/b=1.3$, at $t{\simeq}32.5^\circ$ $X_{59}$ is at the lower self-intersection and the 3-periodic is acute ($X_4$ is interior). \textbf{Right}: at $a/b=\alpha_{59}^\perp\,{\simeq}\,1.58$ the following feat is possible: $X_{59}$ is at the lower self-intersection {\em and} the 3-periodic is a right-triangle ($X_4$ is on $P_2$). This occurs at $t{\simeq}27.9^\circ$. \textbf{Video}: \cite[PL\#12]{reznik2020-playlist-intriguing}}
    \label{fig:x59-center}
\end{figure}

\subsection{Can a Locus be Non-compact}
\label{sec:x26}
$X_{26}$ is the Circumcenter of the Tangential Triangle \cite{mw}. Its sides are tangent to the Circumcircle at the vertices. If the 3-periodic is a right-triangle, its hypotenuse is a diameter of the Circumcircle, and $X_{26}$ is unbounded.

We saw above that:

\begin{itemize}
\item $a/b<\alpha_4$, the 3-periodic family is all-acute, i.e., the locus of $X_{26}$ is compact. Figure~\ref{fig:orthocenter_loci} (top left).
\item $a/b=\alpha_4$, the family is all-acute except when one of its vertices coincides with the top or bottom vertex of the EB, Figure~\ref{fig:orthocenter_loci} (bottom left). In this case the 3-periodic is a right triangle and $X_{26}$ is unbounded.
\item $a/b>\alpha_4$, the family features both acute and obtuse triangles. The transition occurs at for four right-angle 3-periodics whose $X_4$ is on the EB, Figure~\ref{fig:orthic_incenter_locus}(b). Here too $X_{26}$ flies off to infinity, Figure~\ref{fig:x26}.
\end{itemize}

\begin{figure}
    \centering
    \includegraphics[width=\textwidth]{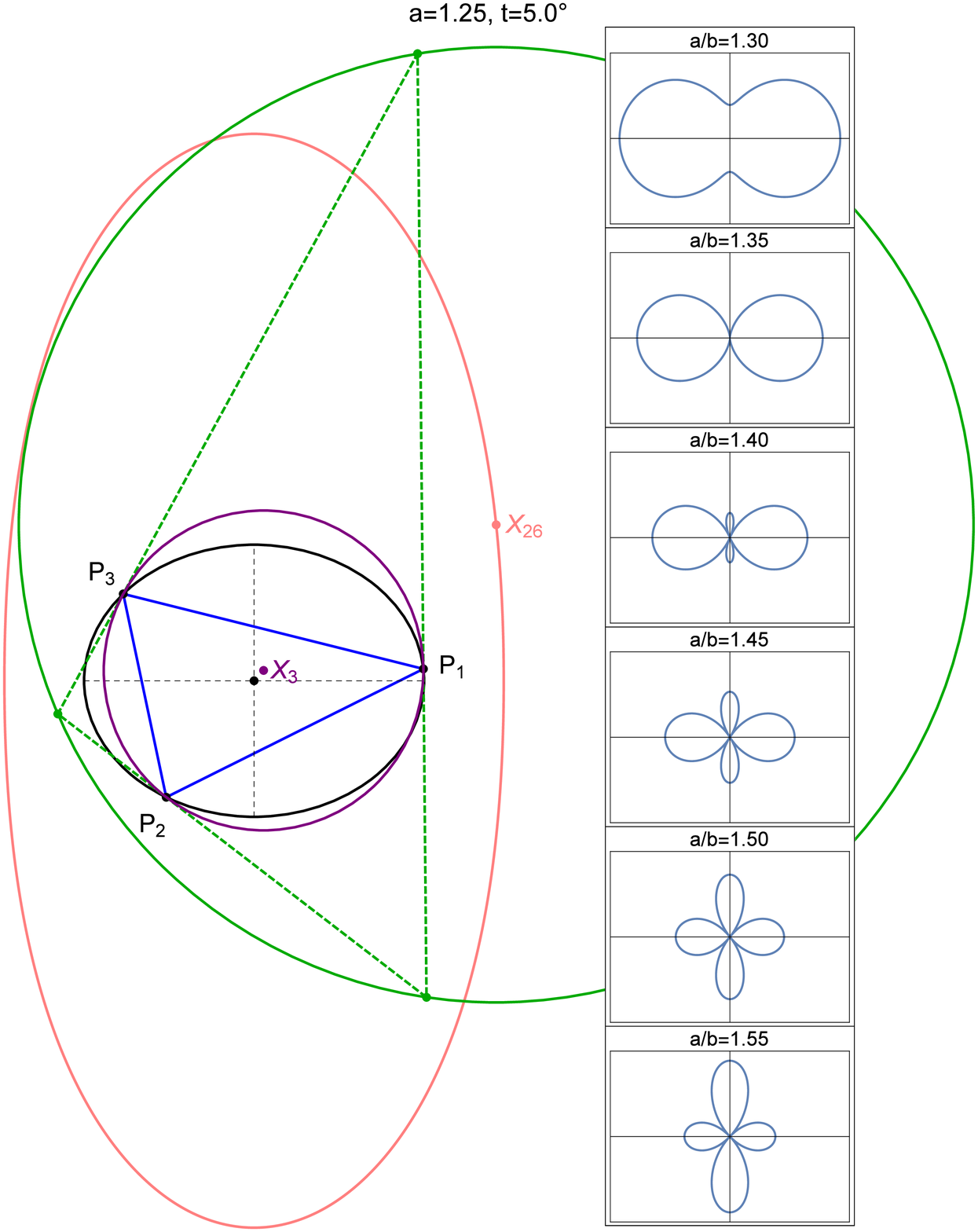}
    \caption{The locus of $X_{26}$ for a 3-periodic (blue) in an $a/b=1.25$ EB (black). Also shown is the 3-periodic's Circumcircle (purple) and its Tangential Triangle \cite{mw} (dashed green). $X_{26}$ is the center of the latter's Circumcircle (solid green). Its locus is non-elliptic. In fact, when $a/b{\geq}\alpha_4$, the 3-periodic family will contain right-triangles ($X_4$ crosses the EB). At these events, $X_{26}$ flies off to infinity. The right inset shows an inversion of $X_{26}$  with respect to the EB center for various values of $a/b$. When $a/b>\alpha_4$, the inversion goes through the origin, i.e., $X_{26}$ is at infinity.}
    \label{fig:x26}
\end{figure}

\section{Intermezzo: Two unexpected phenomena}
\label{sec:intermezzo}
\subsection{The Billiard Lays a Golden Egg}
\label{sec:x40}

The {\em Bevan Point} $X_{40}$ is known as the Circumcenter of the Excentral Triangle \cite{etc}. It is the tangential polygon to the 3-periodic, and can be thought of as its projective dual \cite{levi2007-poncelet-grid}.

We have shown elsewhere the locus of $X_{40}$ is an ellipse similar to a rotated copy of the Billiard. Its semi-axes are given by \cite{garcia2020-ellipses}:

\begin{equation*}
    a_{40}=c^2/a,\;\;\; b_{40}=c^2/b.
\end{equation*}

\begin{proposition}
At $a/b=\sqrt{2}$ i.e., the top and bottom vertices of the $X_{40}$ touch the Billiard's top and bottom vertices.
\end{proposition}

\begin{proof}
This follows from imposing $b_{40}=b$.
\end{proof}

\noindent What we got next was unexpected, see Figure~\ref{fig:x40-golden}:

\begin{proposition}
At $a/b = (1+\sqrt{5})/2=\varphi$, the Golden Ratio, the locus of $X_{40}$ is identical to a $90^\circ$-rotated copy of the EB
\end{proposition}

\begin{proof}
This follows from imposing $b_{40} = a$.
\end{proof}

\begin{figure}[H]
    \centering
    \includegraphics[width=.66\textwidth]{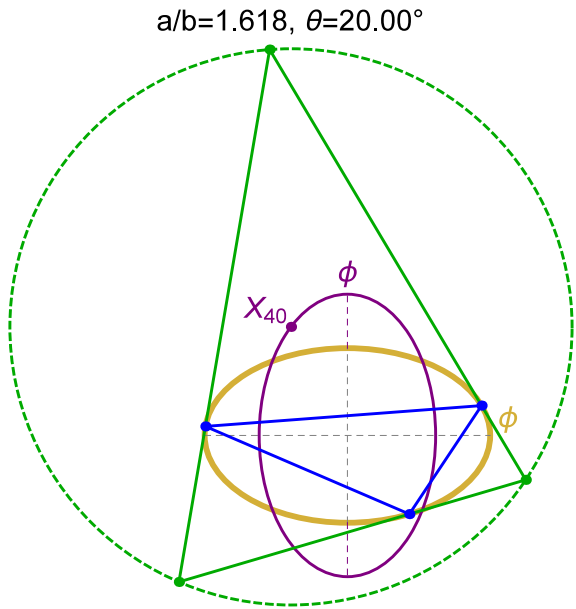}
    \caption{An $a/b=\varphi$ EB is shown golden. Also shown is a sample 3-periodic (blue). The Bevan Point $X_{40}$ is the Circumcenter of the Excentral Triangle (solid green). At this EB aspect ratio, the locus of $X_{40}$ (purple) is identical to a $90^\circ$-rotated EB. \textbf{Video}: \cite[PL\#13]{reznik2020-playlist-intriguing}.}
    \label{fig:x40-golden}
\end{figure}

\subsection{A Derived Triangle railed onto the EB}

Consider a 3-periodic's Anticomplementary Triangle (ACT) \cite{mw} and its Intouchpoints $i_1^\prime,i_2^\prime,i_3^\prime$, Figure~\ref{fig:act_intouch}. Remarkably:

\begin{theorem}
The locus of the Anticomplementary Triangle's Intouchpoints is the EB.
\end{theorem}

\begin{proof}
Consider an elementary triangle with vertices $Q_1=(0,0)$, $Q_2=(1,0)$  and $Q_3=(u,v)$. Its sides are $s_1=|Q_3-Q_2|$, $s_2=|Q_3-Q_1|$, and $s_3=1$.

Let $E$ be its {\em Circumbilliard}, i.e., the Circumellipse for which ${Q_1}{Q_2}{Q_3}$ is a 3-periodic EB trajectory. The following implicit equation for $E$ was derived \cite{garcia2019-incenter}:

\begin{align*}
 E(x,y)=& v^2 x^2+(u^2+(s_1+s_2-1)u-s_2 )y^2+\\
 &v( 1-s_1-s_2-2u )x y+v(s_2+u)y-v^2x =0
\end{align*}

The vertices of the ACT are given by
 $Q_1^\prime=(u-1,v),\,Q_2^\prime=(u+1,v),\,Q_3^\prime=(1-u,-v)$, and its Incenter\footnote{This is the Nagel Point $X_8$ of the original triangle.} is:
 
 \begin{equation*}
 X_1^\prime=\left[s_1-s_2+u, \frac{s_2(s_1-1)+(1-s_1+s_2)u  -u^2}{v}
 \right].
 \end{equation*}
 
 The ACT Intouchpoints are the feet of perpendiculars dropped from $X_1'$ onto each side of the ACT, and can be derived as:
 
 \begin{align*}
 i_1^\prime=& \left[ \frac{s_1(u-1)u+s_2}{s_2},\frac{ v (s_1-1)}{s_2} \right]\\
  i_2^\prime=& \left[  \frac{ (u-1)(s_2-1)}{s_2}, \frac{(s_2-1)v}{s_2}   \right]\\
   i_3^\prime=& \left[   s_1-s_2+u,v \right].
\end{align*}
 
Direct calculations shows that
$E(i_1^\prime)=E(i_2^\prime)=E(i_3^\prime)=0.$. Besides always being on the EB, the locus of the Intouchpoints cover it. Let $P_1(t)P_2(t)P_3(t)$ be a 3-periodic and $P_1^{\prime} (t)P_2^{\prime}(t)P_3^{\prime}(t)$ its ACT.
 For all $t$ the Intouchpoint $i_1^{\prime}(t)$ is located on the side $P_2^{\prime}(t) P_3^{\prime}(t)$ of the ACT and on the elliptic arc     $\textrm{arc}(P_1(t)P_3(t))$, Figure \ref{fig:act_intouch}. Therefore, when $P_1(t)$ completes a circuit on the EB, $i_1^{\prime}(t)$ will have to complete a similar tour. 
Analogously for  $i_2^{\prime}(t)$ and  $i_3^{\prime}(t)$.
\end{proof}

\begin{figure}[H]
    \centering
    \includegraphics[width=\textwidth]{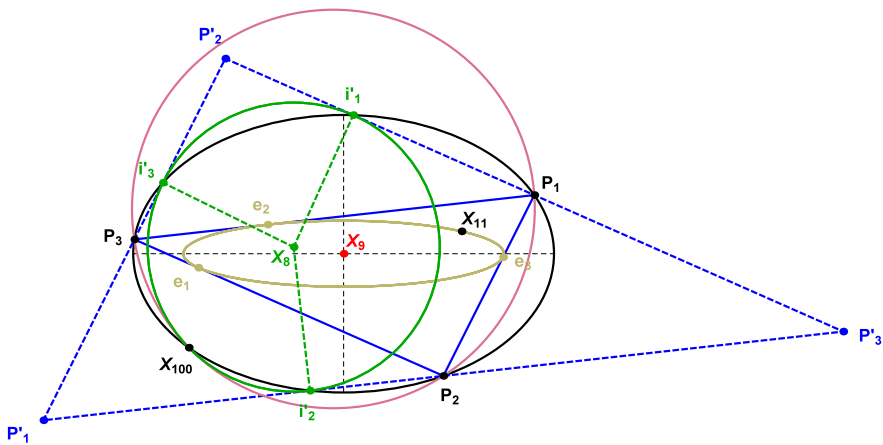}
    \caption{A 3-periodic $P_1P_2P_3$ is shown (blue). Shown also is the Mittenpunkt $X_9$, at the EB center. The 3-periodic's Anticomplementary Triangle (ACT) $P_1'P_2'P_3'$ (dashed blue) has sides parallel to the 3-periodic. The latter's Intouchpoints $i_1'$, $i_2'$, and $i_3'$ are the feet of perpendiculars dropped from the ACT's Incenter ($X_8$) to each side (dashed green). The ACT's Incircle (green) and 9-point circle (the 3-periodic's Circumcircle, pink) meet at $X_{100}$, the ACT's Feuerbach Point. Its locus is also the EB. The Caustic is shown brown. On it there lie $X_{11}$ and the three Extouchpoints $e_1$, $e_2$, $e_3$. \textbf{Video:} \cite[PL\#09]{reznik2020-playlist-intriguing}}
    \label{fig:act_intouch}
\end{figure}

\section{Second Movement: Motion}
\label{sec:act-II}
Let $P_1P_2P_3$ be the vertices of a\ 3-periodic.

\begin{proposition}
 If $P_1$ is slid along the EB in some direction, $P_2$ and $P_3$ will slide in the same direction. 
\end{proposition}

\begin{proof}
Consider the tangency point $C$ of $P_1P_2$ with the confocal Caustic, Figure~\ref{fig:caustic-progress}. Since this segment remains tangent to the Caustic for any choice of $P_1(t)$, a counterclockwise motion of $P_1(t)$ will cause $C$ to slide along the Caustic in the same direction, and therefore $P_2(t)$ will do the same.
\end{proof}

The simultaneous monotonic motion of 3-periodic vertices is shown in Figure~\ref{fig:ekg-p1p2p3}, note the non-linear progress of $P_2,P_3$. Alternatively, we could have linearized their motion using the so-called Poritsky-Lazutkin string length parameter $\eta$ for $C$ on the Caustic (Figure~\ref{fig:caustic-progress}) given by \cite{Poritsky1950, Lazutkin73, alexey19}:

$$
d\eta =  \kappa^{2/3}\,ds,
$$

\noindent where  $s$ is the arc length along the Caustic, and $\kappa$ is the curvature. 
Both $\eta$ and $s$ are related to the parameter $t$ on the billiard by elliptic functions. Adjusting conveniently with a constant factor, one has $\eta\equiv\eta+1$ and for any $\eta_o$ the other vertices correspond to  $\eta_o+1/3$ and $\eta_o+2/3$.

\begin{figure}
    \centering
    \includegraphics[width=.66\textwidth]{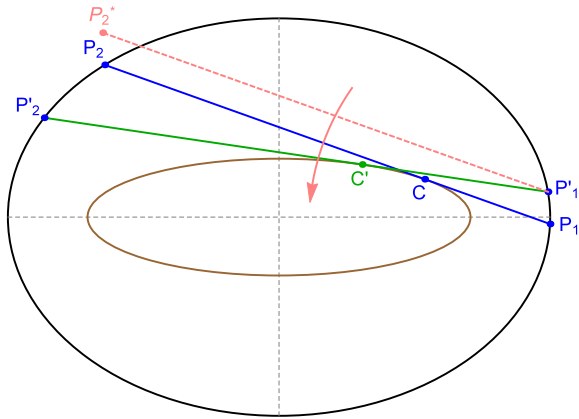}
    \caption{As one endpoint $P_1$ of a billiard trajectory is slid CCW to $P_1'$, its tangency point $C$ with the Caustic (brown) slides in the same direction to $C'$. This must be the case since $P_1'P_2'$ corresponds to a CCW rotation about $P_1'$ of segment $P_1'P_2^*$ (pink) parallel to $P_1P_2$ (see pink arrow). By convexity, said rotation will first touch the Caustic at $C'$, lying ``ahead'' of $C$. Repeating this for the $P_2P_3$ segment of a 3-periodic (not shown), it follows said vertices will move in the direction of $P_1$.}
    \label{fig:caustic-progress}
\end{figure}

\begin{figure}
    \centering
    \includegraphics[width=\textwidth]{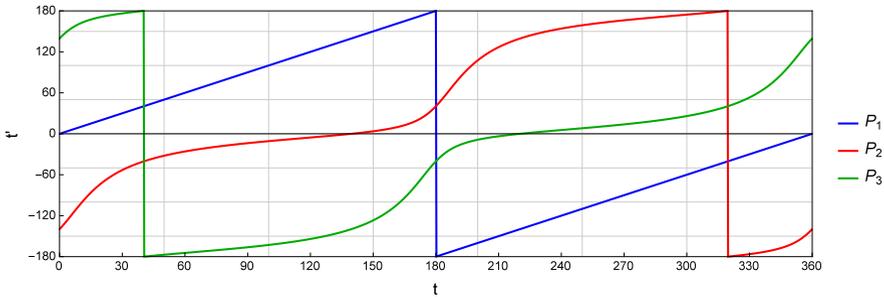}
    \caption{As $P_1(t)$ moves monotonically forward, so do $P_2(t)$ and $P_3(t)$, albeit with varying velocities with respect to $t$. In the text we mention an alternate parametrization (Poritsky-Lazutkin) under which the three lines would become straight.}
    \label{fig:ekg-p1p2p3}
\end{figure}

\subsection{Non-monotonicity: a first brushing}
\label{sec:act}
\mbox{ \,} \\ \smallskip

Let
$\alpha_{act}=2\sqrt{2/5} \simeq 1.2649$.
%$\alpha_{act}{\simeq}1.2649$:

\begin{proposition}
The motion of the ACT Intouchpoints is as follows:

\begin{itemize}
    \item $a/b<\alpha_{act}$: monotonic in the direction of $P_1(t)$.
    \item $a/b=\alpha_{act}$: monotonic in the direction of $P_1(t)$, except for two instantaneous stops when passing at EB top and bottom vertices.
    \item $a/b>\alpha_{act}$:  non-monotonic with four reversals of velocity, a first (resp. second) pair of reversals near the EB's top (resp. bottom) vertex Figure~\ref{fig:act_intouch}.
\end{itemize}
\end{proposition}

\begin{proof}
%The trilinear coordinates of an Intouchpoint are given %by $0:s_1s_3/(s_1-s_2+s_3): s_1s_2/(s_1+s_2-s_3)$ %\cite[Contact Triangle]{mw}.
%
As before, let a 3-periodic $P_1P_2P_3$ be parameterized by a leading vertex $$P_1(t)= (x_1, y_1) = (a\cos t, b\sin t).$$ 
%The ``enslaved'' vertices $P_2(t)$ and $P_3(t)$ are computed explicitly   in \cite{garcia2019-incenter}.  \textcolor{red}{The expressions
%for their coordinates involve double square roots on the coordinate $x_1$ so these points are constructible by ruler and compass\footnote{\textcolor{red}{The parameterization in $t$ is used to generate all the videos in the playlist.}}}.

Its ACT $P_i$' is given by double-length reflections of $P_i$ about the Barycenter $X_2$ \cite{mw}. Taking the ACT as the reference triangle, use Intouchpoint Trilinears $0:s_1s_3/(s_1-s_2+s_3)::$ \cite[Contact Triangle]{mw} to compute an Intouchpoint $i_1'(t)=(x_1(t),y_1(t))$ it follows that
$x_1'(t)\mid_{t=\frac{\pi}{2}}=0 $ is equivalent to $5 a^2-8b^2 =0$. This yields the result. 
\end{proof}

\noindent See \cite[PL\#09]{reznik2019-playlist-math-intelligencer} for an animation of the non-monotonic case. As we had been observing the EKG-like graph in Figure~\ref{fig:x12345-feuer-combo} (right), we stumbled upon an unexpected property, namely, the fixed linear relation between a 3-periodic vertex and its corresponding (opposite) Extouchpoint: 

\begin{proposition}
Let $P_i(t)=(x_i,y_i)$ be one of the 3-periodic vertices and $e_i=(x'_i,y'_i)$ be its corresponding Extouchpoint\footnote{Where $P_{i-1}(t)P_{i+1}(t)$ touch the Caustic.}  on the Caustic, where $a_c,b_c$ are the latter's semi-axes. Then\footnote{It can be shown \eqref{eqn:levi} also holds if $x',y'$ are the coordinates of an Excenter and $a_c,b_c$ are the semi-axes of the excentral locus, known to be an ellipse \cite{garcia2019-incenter}.}, for all $t$:

\begin{equation}
\frac{a}{b}\frac{y_i}{x_i}=\frac{a_c}{b_c}\frac{y'_i}{x'_i}
\label{eqn:levi}
\end{equation}

\noindent Equivalently, for $e_i(t')=[a_c \cos{t'},b_c\sin{t'}]$, then $\tan(t)=\tan(t')$, i.e., $t'=t{\pm}\pi$.
\label{lem:extouch}
\end{proposition}

\begin{proof}
This property follows directly\footnote{We thank A. Akopyan for pointing this out.} from \cite[Lemma 3]{levi2007-poncelet-grid}.
\end{proof}

In fact,  more general properties of the 
``Poncelet grid'' are described in 
the reference above. The  result reported here is a particular case and can also be demonstrated by simplifying rather long symbolic parametrics with a Computer Algebra System (CAS).

Furthermore, since  $a_c=(\delta-{b}^{2})a/c^2$ and $b_c=({a}^{2}-\delta)b/c^2$ \cite{garcia2020-new-properties}: 

\begin{equation}
\frac{y}{x}=\left(\frac{\delta-b^2}{a^2-\delta}\right)\frac{y'}{x'}
\end{equation}

\begin{figure}[H]
    \centering
    \includegraphics[width=\textwidth]{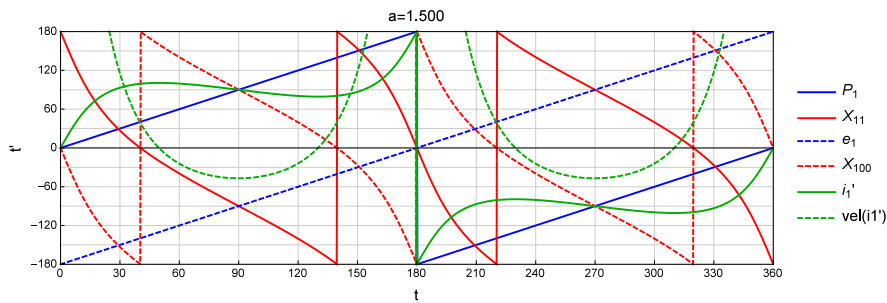}
    \caption{The ``EKG'' of ACT motion, to be interpreted as a flat torus. On the horizontal axis is parameter $t$ of $P_1(t)=[a\cos(t),b\sin(t)]$, with $a/b=1.5$. On the vertical is the $t$ parameter for $P_1,X_{11},X_{100}$ and an ACT Intouchpoint $i_1'$. These all lie on the EB and are shown blue, red, dashed red, and green, respectively. Also shown is Extouchpoint $e_1$ on the Caustic (dashed blue). Just for it, the vertical axis represents $t'$ in $e_1=[a_c\cos(t'),b_c\sin(t')]$, where $a_c,b_c$ are the Caustic semi-axes. By Proposition~\ref{lem:extouch}, $t'=t{\pm}\pi$. Notice the only non-monotonic motion is that of $i_1'$ since $a/b>\alpha_{act}{\simeq}1.265$. To see this, a $\text{vel}(i_1')$ of its velocity is also shown (dashed green), containing two negative regions corresponding to 4 critical points of position. For $vel(i_1')$ ignore units and the fact that values near $0^\circ,180^\circ$ are not shown, these are all positive and above the vertical scale.}
    \label{fig:act-progress}
\end{figure}

\subsection{A Non-Monotonic Triangle Center}
\label{sec:x88}
Dovetailing into the non-monotonicity of the ACT's Intouchpoints is a similar behavior by $X_{88}$, the Isogonal Conjugate of $X_{44}$, known to lie on the EB and to be to be collinear with $X_1$ and $X_{100}$, \cite{etc}. The latter is verified by the vanishing determinant of the 3x3 matrix whose rows are the trilinears of $X_1,X_{100},X_{88}$  \cite[under ``Collinear'', eqn 9]{mw}:

\[
\det\left[
\begin {array}{ccc}
1&1&1\\
 \frac{1}{s_2-s_3} & \frac{1}{s_3-s_1} & \frac{1}{s_1-s_2} \\ 
\frac{1}{s_2+s_3-2\,s_1}& \frac{1}{s_1+s_3-2\,s_2} & \frac{1}{s_1+s_2-2\,s_3}
\end{array}
\right]{\equiv}\,0
\]

\smallskip

Furthermore $X_{100}$ is a very special point: it lies on the EB and on the Circumcircle simultaneously \cite{etc}. Let $\alpha_{88}=(\sqrt{6+2\sqrt{2}}\,)/2\simeq{1.485}$.

\begin{proposition}
At $a/b=\alpha_{88}$, the y velocity of $X_{88}$ vanishes when the 3-periodic is a sideways isosceles.
\end{proposition}

\begin{proof}
Parametrize $P_1(t)$ in the usual way. At $t=0$, $P_1=(a,0)$ it can be easily checked that $X_{88}=(-a,0)$. Solve $y_{88}'(t)|_{t=0}=0$ for $a/b$. After some algebraic manipulation, this equivalent to solving $4x^4-12x^2+7=0$, whose positive roots are $(\sqrt{6\pm 2\sqrt{2}}\,)/2$. $\alpha_{88} $ is the largest of the two.
\end{proof}

As shown in Table~\ref{tab:x88}, there are three types of $X_{88}$ motion with respect to $P_1(t)$: monotonic, with stops at the EB vertices, and non-monotonic.

{
\begin{table}[!htbp]
\begin{tabular}{|c|l|l|}
\hline
$a/b$ vs. $\alpha_{88}$ & motion & comment \\
\hline
$<$ & CW & monotonic\\
$=$ & CW & stops at EB vertices \\
$>$ & CW+CCW & non-monotonic \\
\hline
\end{tabular}
\caption{Conditions for the type of motion of $X_{88}$ with respect to $P_1(t)$. \textbf{Video}: \cite[PL\#11]{reznik2020-playlist-intriguing}}
\label{tab:x88}
\end{table}
}

An equivalent statement is that the line family $X_1X_{100}$ is instantaneously tangent to its {\em envelope} \cite{mw} at $X_{88}$. Figure~\ref{fig:x88-envelope} shows that said envelope lies (i) entirely inside, (ii) touches at vertices of, or (iii) is partially outside, the EB, when $a/b$ is less than, equal, or greater than $\alpha_{88}$, respectively. Each such case implies the motion of $X_{88}$ is (i) monotonically opposite to $P_1(t)$, (ii) opposite but with stops at the EB vertices, or (iii) is non-monotonic.

The reader is challenged to find an expression for parameter $t$ in $P_1(t)$ where the motion of $X_{88}$ changes direction. The following additional facts are also true for $X_{88}$:

\begin{proposition}\label{prop:x88}
$X_{88}$ coincides with a 3-periodic vertex if and only if $s_2 = (s_1+s_3)/2$. In this case, $X_1$ is the midpoint between $X_{100}$ and $X_{88}$ \cite{helman20}, Figure~\ref{fig:x88tri} (bottom left)
\end{proposition}

\begin{proof}
The first trilinear coordinate of $X_{88}$ is $1/(s_2 + s_3 - 2 s_1)$ \cite{etc}, and of a vertex is $0$. Equating the two yields $s_2 = (s_1+s_3)/2$.
Consider a triangle of reference $P_1=(-1,0)$, $P_2=(u,v)$, $P_3=(1,0)$. Its circumcircle is given by  $v(x^2+y^2)+(1-u^2-v^2)y-v=0.$ Under the hypothesis $s_2=(s_1+s_3)/2$ it follows that $v=\sqrt{12-3u^2}/2$, $s_1=2-u/2$, $s_2=2$ and $s_3=2+u/2$.  Therefore the incenter is   $ I= (s_1P_1+s_2P_2+s_3P_3)/(s_1+s_2+s_3)=(u/2, \sqrt{12-3u^2}/6)$. The intersection of the straight line passing through $P_2$ and $I$ with the circumcircle of the triangle o reference is the point $D=(0,\sqrt{12-3u^2}/6).$
Therefore, $I$ is the midpoint of $B$ and $D=X_{100}$. Moreover, $\mid P_1-D\mid=\mid P_2-D\mid=\mid I-D\mid= \sqrt{48-3u^2}/6$.
\end{proof}

\begin{figure}
    \centering
    \includegraphics[width=\textwidth]{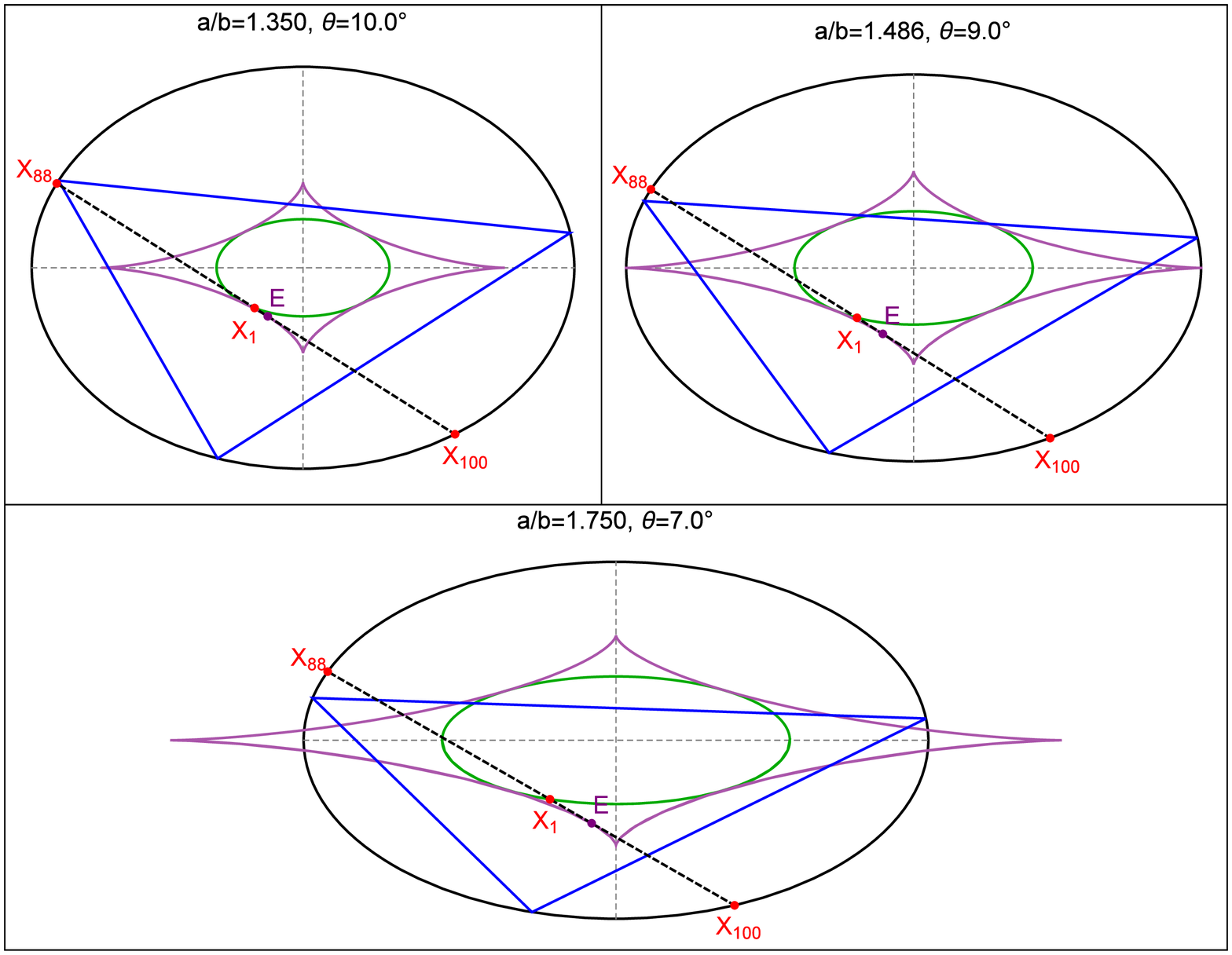}
    \caption{Collinear points $X_1,X_{100},X_{88}$ shown for billiards with $a/b$ less than (top-left), equal (top-right), or greater (bottom) than $\alpha_{88}{\simeq}1.486$, respectively. In each such case, the motion of $X_{88}$ relative to the 3-periodic vertices will be monotonic, with stops at the vertices, or non-monotonic, respectively. Equivalently, the motion of $X_{88}$ is opposite to $P_1$, stationary, or in the direction of $P_1$ if the instantaneous center of rotation $E$ of line $X_1X_{100}$ lies inside, on, or outside the EB. The locus of $E$ (the envelope of $X_1X_{100}$) is shown purple. Notice it only ``pierces'' the EB when $a/b>\alpha_{88}$ (bottom), i.e., only in this case can the motion of $X_{88}$ be non-monotonic. \textbf{Video}: \cite[PL\#12]{reznik2020-playlist-intriguing}}
    \label{fig:x88-envelope}
\end{figure}

It is well-known that the only right-triangle with one side equal to the average of the other two is $3:4:5$. Let $\alpha_{88}^\perp=(7+\sqrt{5})\sqrt{11}/22\simeq{1.3924}$. Referring to Figure~\ref{fig:x88tri} (right):

\begin{proposition}
The only EB which can contain a 3:4:5 3-periodic has an aspect ratio $a/b=\alpha_{88}^\perp$.
\end{proposition}

\begin{proof}
With $a/b>\alpha_4=\sqrt{2\,\sqrt{2}-1}\simeq{1.352}$ the 3-periodic family contains obtuse triangles amongst which there always are 4 right triangles (identical up to rotation and reflection). Consider the elementary triangle $P_1=(0,0)$, $P_2=(s_1,0)$, and $P_3=(s_1,s_2)$ choosing $s_1,s_2$ integers such that $s_3=s_3=\sqrt{s_1^2+s_2^2}$ is an integer. The Circumbilliard \cite{garcia2020-new-properties} is given by:

\[
E_9(x,y)= s_2 x^2+ \left( s_3-s_1-s_2 \right) xy+s_1{y}^{2}-s_1s_2x-s_1 \left( s_1-s_3 \right) y=0.\]

Squaring the ratio of the Eigenvalues of $E_9$'s Hessian yields the following expression for $a/b$:

\begin{equation}
    [a/b](s_1,s_2,s_3)= \frac {s_1+s_2+\sqrt {s_3\left(3\,s_3-2\,s_1-2\,s_2\right) }}{\sqrt { \left( 
s_1+s_2+3\,s_3 \right)  \left( s_1+s_2-s_3 \right) }}
\end{equation}
\end{proof}

%The sides $s_1,s_2,s_3$ of a right triangle are related as %sides $s_1,s_2,s_3$ for which $s_1,s_2$ are co-prime, and %$s_3^2=s_1^2+s_2^2$.

\begin{figure}
    \centering
    \includegraphics[width=\textwidth]{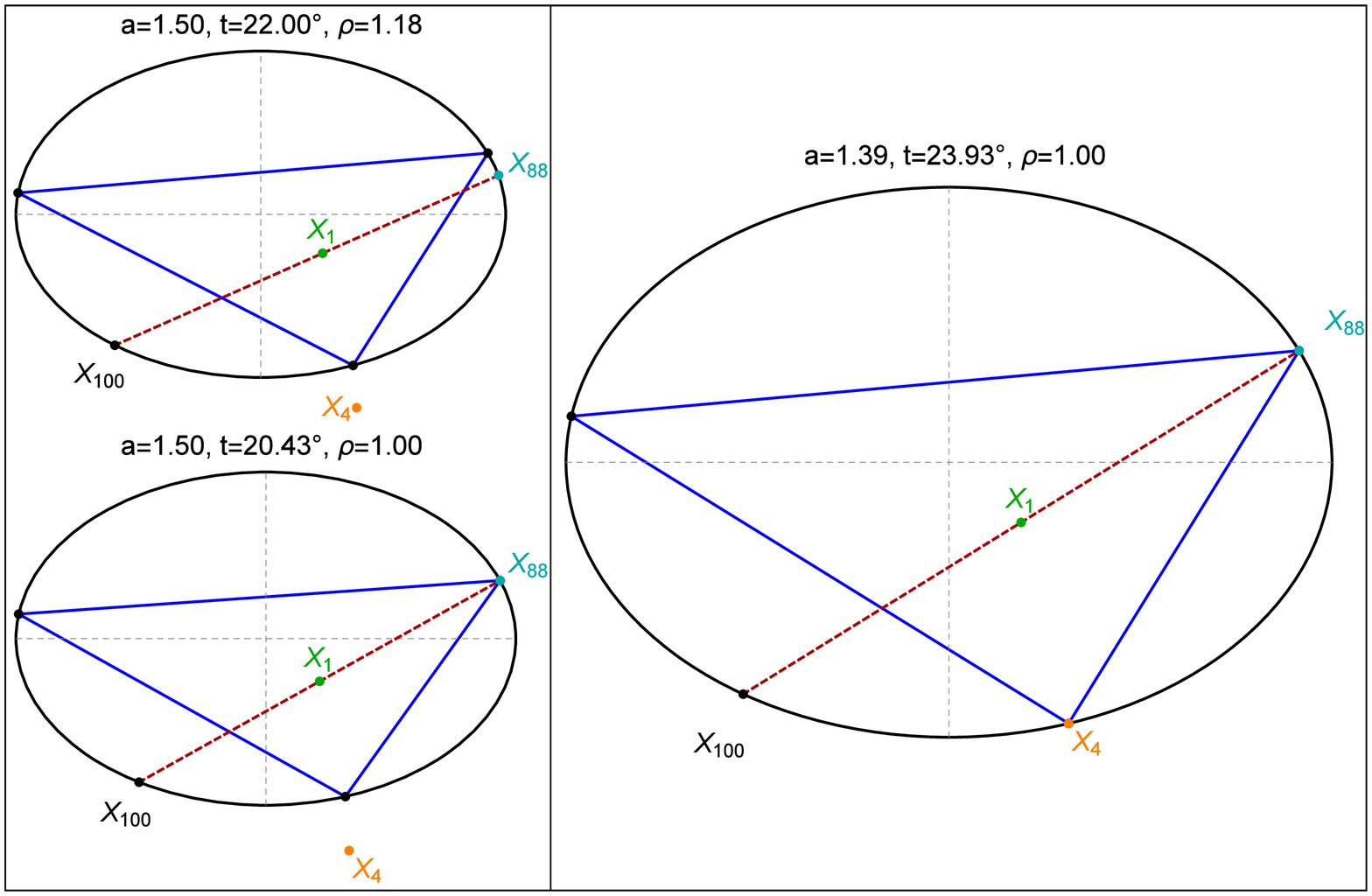}
    \caption{$X_{88}$ is always on the EB and collinear with $X_1$ and $X_{100}$ \cite{etc}. Let $\rho$ (shown above each picture) be the ratio $|X_1-X_{100}|/|X_1-X_{88}|$. \textbf{Top Left}: The particular 3-periodic shown is obtuse ($X_4$ is exterior), and $\rho>1$, i.e., $X_1$ is closer to $X_{88}$. \textbf{Bottom Left}: When $X_{88}$ coincides with a vertex, if sidelengths are ordered as $s_1{\leq}s_2{\leq}s_3$, then $s_2=(s_1+s_3)/2$, and $X_1$ becomes the midpoint of $X_{88}X_{100}$, i.e., $\rho=1$. \textbf{Right}: If $a/b=\alpha_{88}^\perp\,{\simeq}\,1.39$, when $X_{88}$ is on a vertex, the 3-periodic is a $3:4:5$ triangle ($X_4$ lies on an alternate vertex). \textbf{Video}: \cite[PL\#11]{reznik2020-playlist-intriguing}}
    \label{fig:x88tri}
\end{figure}

\noindent Table~\ref{tab:pythagorean} shows $a/b$ for the first 5 Pythagorean triples ordered by hypotenuse\footnote{The $a/b$ which produces $3:4:5$ was first computed in connection with $X_{88}$ \cite{helman20-private-345}.}.

\begin{table}[!h]
$$
\scriptsize
\begin{array}{r|c|l|l}
(s_1,s_2,s_3) & a/b & {\simeq}a/b \\
\hline
3, 4, 5 & {(7+\sqrt{5})\sqrt{11}}/{22} &  1.392  \\
5, 12, 13 & {\sqrt{14} (\sqrt{65}+17)}/{56} & 1.674  \\
8, 15, 17 & {\sqrt{111} (\sqrt{85}+23)}/{222} & 1.529 \\
7, 24, 25  & {\sqrt{159} (5\,\sqrt{13}+31)}/{318} & 1.944 \\
20, 21, 29 & \sqrt{6} (\sqrt{145}+41)/{96} & 1.353 &   \\
\end{array}
$$
%\[a/b_{\perp}(m,n,p)=\sqrt {{\frac {m+n+\sqrt {p %( -2\,m-2\,n+3\,p \right) }}{m+n-
%\sqrt {p ( -2\,m-2\,n+3\,p \right) }}}}\]
\caption{First 5 Pythagorean triples ordered by hypotenuse. $a/b$ is the aspect ratio.
of the EB which produces 4 triangles homothetic to the triple.}
\label{tab:pythagorean}
\end{table}

\subsection{Swan Lake}
\label{sec:swan-lake}
\noindent
\epigraph{But now they drift on the still water, \\
Mysterious, beautiful; \\
Among what rushes will they build, \\
By what lake's edge or pool \\
Delight men's eyes when I awake some day \\
 To find they have flown away?}{\textit{W.B. Yeats}}

In addition to $X_{88}$ and $X_{100}$, a gaggle of 50+ other TCs are also known to lie on the EB \cite[X(9)]{etc}, e.g., $X_{162}$, $X_{190}$, $X_{651}$, etc. We have found experimentally that the motion of $X_{100}$ and $X_{190}$ is always monotonic.

In contradistinction: Let $\alpha_{162}{\simeq}1.1639$ be the only positive root of $5 x^8 + 3 x^6 - 32 x^4 + 52 x^2 - 36$.

\begin{proposition}\label{prop:x162}
The motion of $X_{162}$ with respect to $P_1(t)$ is non-monotonic if $a/b>\alpha_{162}$. 
\end{proposition}

\begin{proof}
The trilinear coordinates of $X_{162}$ are given by
{\small 
\[  \frac {1}{ \left( s_2^{2}-  s_3^{2} \right)  \left( s_2^{2}+ s_3^{2}-
s_1^{2} \right) } :  \frac {1}{ \left( s_3^{2}-  s_1^{2} \right)  \left( s_3^{2}+ s_1^{2} -
s_2^{2} \right) }:\frac {1}{ \left( s_1^{2}-  s_2^{2} \right)  \left( s_1^{2}+ s_3^{2} -
s_3^{2}\right) }\cdot
\]
}
Let a 3-periodic $P_1P_2P_3$ be parametrized by $P_1(t)=(a\cos t, b\sin t)$, with $P_2(t)$ and $P_3(t)$ computed explicitly as in \cite{garcia2019-incenter}.
Using the trilinear coordinates above, we have   $X_{162}(t)=(x_{126}(t), y_{126}(t))$ 
  At $t=\frac{\pi}{2}$, $P_1=(0,b)$ and   $X_{162}(\frac{\pi}{2})= (0,b)$.
  
  Solve $x_{162}'(t)|_{t=\frac{\pi}{2}}=0$ for $a/b$. After some long algebraic symbolic manipulation, this is equivalent to solving $5 x^8 + 3 x^6 - 32 x^4 + 52 x^2 - 36=0$, whose positive roots is $  \alpha_{162} \simeq 1.16369.$

\end{proof}

Since $\alpha_{88}>\alpha_{162}$, settting $a/b>\alpha_{88}$ implies both centers will move non-monotonically.

If the EB be a lake, their joint motion is a dance along its margins.
Over a complete revolution of $P_1(t)$ around the EB, $X_{88}$ and $X_{162}$ wind thrice around it, however:

\begin{proposition}
$X_{88}$ and $X_{162}$ never coincide, therefore their paths never cross each other.
\end{proposition}

\begin{proof} Consider an elementary triangle $P_1=(-1,0)$, $P_2=(1,0)$ and $P_3=(u,v).$
Using the trilinear coordinates of $X_{88} $ and $X_{162}$, see Propositions \ref{prop:x88} and \ref{prop:x162}, and equation  \eqref{eqn:trilin-cartesian} compute the triangle centers $X_{88} $ and $X_{162}$.    The  equation $X_{88}=X_{162}$ is given by two algebraic equations $F(u,v,s_1,s_2)=G(u,v,s_1,s_2)=0$ of degree 17 with   $s_1=\sqrt{(u-1)^2+v^2}=\mid P_3-P_2\mid$ and $s_2=\sqrt{(u+1)^2+v^2}=\mid P_2-P_1\mid$.
Particular solutions of these equations are    equilateral triangles with $P_3=(0,\pm \sqrt{3}).$ 
Analytic and  graphic analysis reveals that the level curves $F=G=0$ are as shown in Figure~\ref{fig:x88x162}. 
\begin{figure}
    \centering
    \includegraphics[width=.65\textwidth]{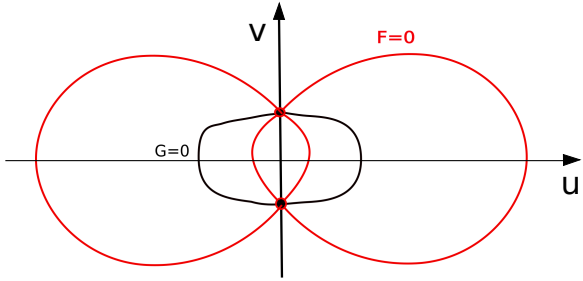}
    \caption{Level curves $F=0$ (red) and $G=0$ (black). }
    \label{fig:x88x162}
\end{figure}

Therefore the equilateral triangle is the only one such that    the triangle centers $X_{88}$ and $X_{162}$ are equal.  It is well known that an equilateral billiard orbit occurs only  when the billiard ellipse is a circle. This ends the proof.
%\textcolor{red}{sketch. A figura esta com %cores erradas. Vou corrigir mas nao consigo %agora pois o meu inkscape nao funciona mais %com a atualizacao do mac.}
\end{proof}

Their never-crossing joint motion as well as the instants when they come closest is illustrated in Figure~\ref{fig:flat-torus}.

\begin{figure}[H]
    \centering
    \includegraphics[width=\textwidth]{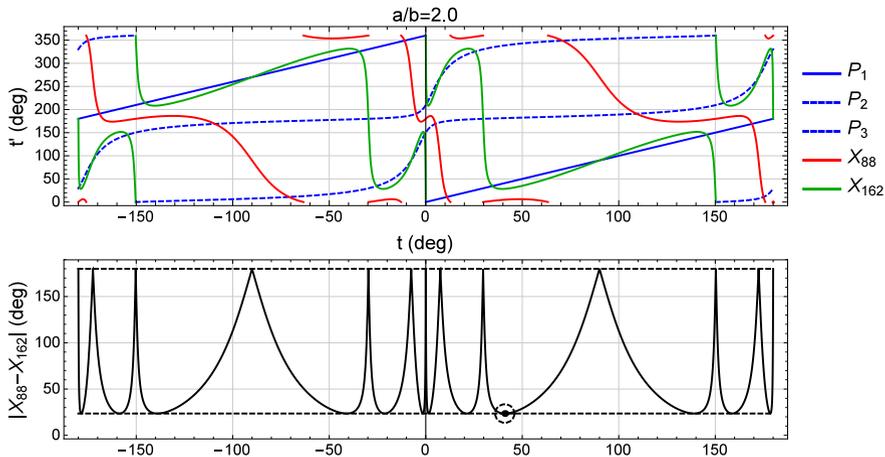}
    \caption{\textbf{Top}: location $t'$ of 3-periodic vertices $P_1$ (blue), $P_2$ (dashed blue), and $P_3$ (dashed blue), as well as $X_{88}$ (red), and $X_{162}$, plotted against the $t$ parameter of $P_1(t)$. \textbf{Bottom}: absolute parameter difference along the EB between $X_{88}$ and $X_{162}$. Notice there are 12 identical maxima at $180^\circ$ occurring when the two centers at the left and right vertices of the EB. Additionally, there are 12 identical minima whose values can be obtained numerically. The fact that the minimum is above zero implies the points never cross. Note the highlighted minimum at $t{\simeq}41^\circ$: it is referred to in Figure~\ref{fig:3d-torus}.}
    \label{fig:flat-torus}
\end{figure}

These paths can also be viewed as non-intersecting loops on the torus shown in Figure~\ref{fig:3d-torus}. A part of the harmonious ballet of $X_{88}$ and $X_{162}$ is depticted in Figure~\ref{fig:x88-x162} with each center imagined a swan.

\begin{figure}
    \centering
    \includegraphics[width=.8\textwidth]{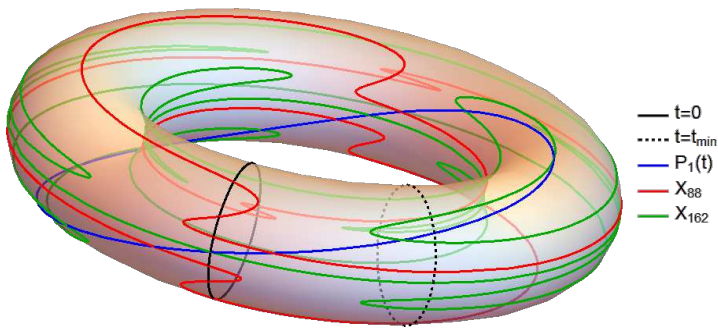}
    \caption{Visualizing the joint motion of $P_1(t),X_{88},X_{162}$ on the surface of a torus. The meridians (circles around the smaller radius) correspond to a given $t$ (a solid black meridian is wound at $t=0$. The parallels represent a fixed location on the Billiard boundary. The curves for $X_{88}$ and $X_{162}$ are thrice-winding along the torus though never intersecting. To be determined: an analytic value for $t$ where $X_{88}$ and $X_{162}$ are closest (there are 12 solutions). The dashed meridian represents one such minimum which for $a/b=2$ occurs at $t{\simeq}41^\circ$. Notice it does not coincide with any critical points of motion.}
    \label{fig:3d-torus}
\end{figure}

\begin{figure}
    \centering
    \includegraphics[width=\linewidth]{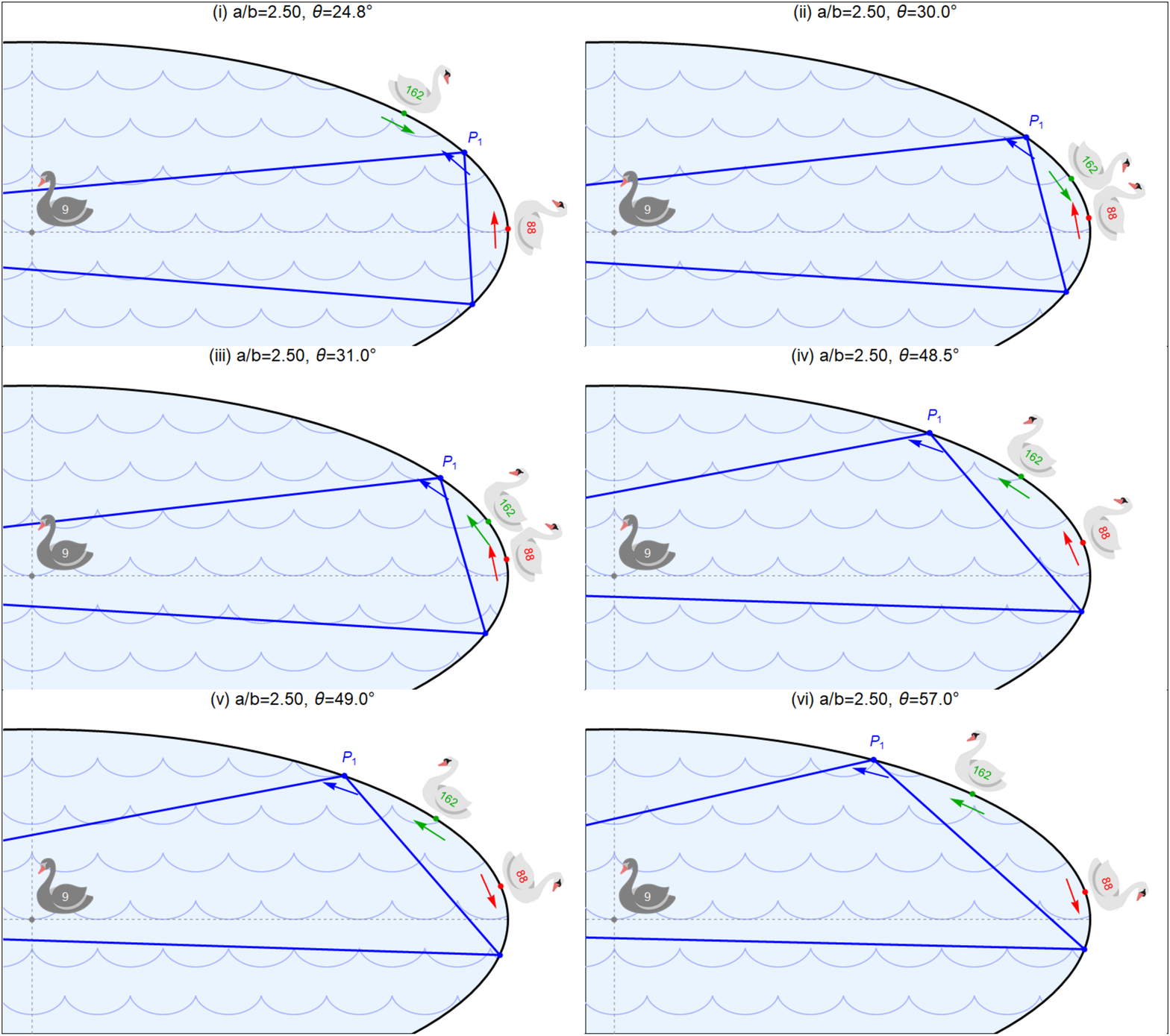}
    \caption{Triangle Center Ballet along the margins of an $a/b=2.5$ Elliptic Lake. (i) while $P_1$ moves CCW, Triangle Swanters $X_{88}$ and $X_{162}$ glide toward each other; (ii) at their closest they touch bills. (iii) Suddenly, $X_{162}$ reverses course, (iv) and a short-lived same-direction pursuit ensues. (v) An unswooned $X_{88}$ also changes course, (vi) and now both glide in opposite directions. The duo will meet again on 2nd, 3rd and 4th quadrants, where the dance steps are played back in alternating forward and backward order. Unfazed, a black {\em Mittenschwan} floats still at the center of the Lake. With some help from Tchaikovsky, \textbf{Video}: \cite[PL\#14]{reznik2020-playlist-intriguing}}
    \label{fig:x88-x162}
\end{figure}

Though we lack a theory for non-monotonicity of EB-bound TCs $X_i$, we think it is ruled by at least the following aspects:

\begin{itemize}
    \item When the 3-periodic is an isosceles, will $X_i$ lie at the summit vertex or below the base?
    \item As the 3-periodic rotates about the EB, does $X_i$ follow it or move counter to it?
    \item Can $X_i$ be non-monotonic? For example, neither $X_{100}$ nor $X_{190}$ ever are. If so, what is the aspect ratio $\alpha_i$ which triggers it?
\end{itemize}

Still murky is how the above derive from the Trilinears which specify $X_i$. We refer the reader to an animation depicting the joint motion of 20-odd such centers \cite[PL\#15]{reznik2020-playlist-intriguing}.

\subsection{Summary of Phenomena}
\label{sec:inter-summary}

Tables~\ref{tab:loci-phenomena} and \ref{tab:loci-ab} respectively summarize this Section's various loci phenomena, and notable $a/b$ thresholds.

\begin{table}[H]
\scriptsize
\begin{tabular}{|r|l|c|l|}
\hline
Point & Type of Locus & Non-Monotonic & Comments \\
\hline
%$X_1$ & Ellipse & Limit of curve with two-internal loops \\
$X_{11}$ & Caustic & -- & reverse dir. \\
Extouchpts. & Caustic & -- & forward dir. \\
$X_4$ & Upright Ellipse & -- & Inside EB when $a<\alpha_4$ \\
%$X_6$ & Convex Quartic & Externally-Tangent to ellipse $(a_6,b_6)$, Appendix~\ref{app:symmedian} \\
Orthic $X_1$ & Can be 4-piece Ell. & -- &
\makecell[tl]{If $a<\alpha_4$ is $X_4$ locus,\\else 4-piecewise ellipse:\\\;\;2 arcs from $X_4$ locus,\\\;\;2 arcs from EB} \\

$X_{26}$ & Can be non-compact & -- & \makecell[tl]{When $a/b{\geq}\alpha_4$, goes to $\infty$\\if 3-periodic is right triangle} \\
$X_{40}$ & Upright Ellipse & -- & At $a/b=\varphi$ identical to EB \\
$X_{59}$ & At Least Sextic & -- & 4 Self-Intersections \\
$X_{88}$ & EB & $a/b>\alpha_{88}$ & \makecell[tl]{at $a/b=\alpha_{88}^\perp>\alpha_4$\\contains 3:4:5 triangle}\\
$X_{162}$ & EB & $a/b>\alpha_{162}$ & never crosses $X_{88}$ \\
ACT Intouchpts. & EB Boundary & $a/b>\alpha_{act}$ & \\
\hline
\end{tabular}
\caption{Loci phenomena for various Triangle Centers.
\label{tab:loci-phenomena}}\end{table}

\begin{table}[H]
\scriptsize
\begin{tabular}{|c|l|c|l|}
\hline
symbol & $a/b$ & Degree & Significance \\
\hline
$\alpha_{162}$ & $1.164$ & 8 & above it, motion of $X_{162}$ is non-monotonic  \\ 
$\alpha_h$ & $1.174$ & 8 & \makecell[tl]{above it, some 3-periodic Orthics \\ can be obtuse.} \\ 
$\alpha_{act}$ & $1.265$ & 2 & \makecell[tl]{above it, motion of ACT\\Intouchpoints is non-monotonic} \\
$\alpha_4$ & $1.352$ & 4 & \makecell[tl]{locus of $X_4$ is tangent to EB.\\above it, some 3-periodics are obtuse} \\
$\alpha_{88}^{\perp}$ & $1.392$ & 4 & with $X_{88}$ on a vertex, sidelengths are $3:4:5$ \\
$\alpha_{88}$ & $1.486$ & 4 & above it, motion of $X_{88}$ is non-monotonic \\
$\alpha_4^{*}$ & $1.510$ & 6 & $X_4$ locus identical to rotated EB \\
$\alpha_{59}^{\perp}$ & $1.580$ & -- & \makecell[tl]{when $X_{59}$ is at self-intersection,\\3-periodic is right triangle}\\
$\varphi$ & $1.618$ & 2 & $X_{40}$ locus identical to rotated EB\\
\hline
\end{tabular}

\caption{Aspect ratio thresholds, the degree of the polynomial used to find them, and their effects on loci phenomena. \label{tab:loci-ab}}
\end{table}

\section{Conclusion}
\label{sec:conclusion}
Examining the mysterious, beautiful 3-periodics under the lens of Classical Triangle Geometry has yielded many a cygnet. We think the EB's chant is far from over, surely many delightful secrets still hide under its plumage.

Still, our haphazard experimental process could use some theoretical teeth. As next steps, we submit the following questions to the reader:

\begin{itemize}
    \item Can a Triangle Center be found such that its locus can intersect a straight line more than 6 times?
    \item The questions about $X_{59}$ in Section~\ref{sec:x59}.
    \item What causes a Triangle center to move monotonically (or not), forward or backward, with respect to the the monotonic motion of 3-periodic vertices?
    \item What is $t$ in $P_1(t)$ at which point $X_{88}$ or $X_{162}$ reverses motion? When do they come closest?
    \item What can be said about the joint motion of other pairs in the 50+ list of EB-bound points provided in \cite[X(9)]{etc}? Which are monotonic, which are not, is there a Pavlova or Baryshnikov amongst them?
    \item With \cite{reznik2019-applet} one observes $X_{823}$ reverses direction at the exact moment a 3-periodic vertex crosses it. Can this be proven?
\end{itemize}

Videos mentioned herein are on a \href{https://bit.ly/2vvJ9hW}{playlist} \cite{reznik2020-playlist-intriguing}, with links provided on Table~\ref{tab:playlist}. The reader is especially encouraged to interact with the phenomena above using our online applet \cite{reznik2019-applet}. A gallery of loci generated by $X_1$ to $X_{100}$ (as well as vertices of derived triangles) is provided in \cite{reznik2019-loci-gallery-all-derived-tris}. 

\begin{table}[H]
\small
\begin{tabular}{|c|l|l|}
\hline
\href{https://bit.ly/2vvJ9hW}{PL\#} & Video Title & Section\\
\hline
\href{https://youtu.be/tMrBqfRBYik}{01} &
{$X_9$ stationary at EB center} & \ref{sec:intro} \\
\href{https://youtu.be/sMcNzcYaqtg}{02} &
{Loci for $X_1\ldots{X_5}$ are ellipses} &
\ref{sec:intro} \\
\href{https://youtu.be/Xxr1DUo19_w}{03} &
{Elliptic locus of Excenters similar to rotated $X_1$} &
\ref{sec:intro} \\
\href{https://youtu.be/TXdg7tUl8lc}{04} &
{Loci of $X_{11}$, $X_{100}$ and Extouchpoints are the EB} &\ref{sec:intro} \\

\href{https://youtu.be/OGvCQbYqJyI}{05} &
{Non-Ell. Loci of Medial, Intouch and Feuerbach Vertices} & \ref{sec:intro} \\

\href{https://youtu.be/-bLuvICzmqM}{06} &
{Pinning of Incenter of Orthic to Obtuse Vertex} & \ref{sec:x4} \\

\href{https://youtu.be/3qJnwpFkUFQ}{07} &
{Locus of Orthic Incenter is 4-piece ellipse} &
\ref{sec:orthic-incenter} \\

\href{https://youtu.be/HY577AZVi7I}{08} &
Locus of $X_4$, Orthic $X_1,X_4$, and Orthic's Orthic $X_1$ &
\ref{sec:orthic-incenter} \\

\href{https://youtu.be/50dyxWJhfN4}{09} &
{Non-monotonic motion on the EB of Anticompl. Intouchpoints} &
\ref{sec:act} \\
\href{https://youtu.be/DaoNJRcf-0E}{10} &
{Locus of $X_{88}$ is on the EB and can be non-monotonic} &
\ref{sec:x88} \\
\href{https://youtu.be/nJLp--JjDZU}{11} &
{Non-monotonic motion of $X_{88}$ and $X_1X_{100}$ envelope} &
\ref{sec:x88} \\
\href{https://youtu.be/pl_PqSuhlx0}{12} &
{Locus of $X_{59}$ with 4 self-intersections} &
\ref{sec:x59} \\
\href{https://youtu.be/rg28gGr-Qeo}{13} &
{Locus of $X_{40}$ and the Golden Billiard} & \ref{sec:x40}\\
\href{https://youtu.be/ljGTtA1x-Sk}{14} &
{Swan Lake: the dance of $X_{88}$ and $X_{162}$} &
\ref{sec:swan-lake}  \\
\href{https://youtu.be/JdcJt5PExsw}{15} &
{Peter Moses' Points on the EB} &
\ref{sec:swan-lake} \\
\hline
\end{tabular}
\caption{Videos mentioned in the paper. Column ``PL\#'' indicates the entry within the \href{https://bit.ly/2vvJ9hW}{playlist} \cite{reznik2020-playlist-intriguing}.}
\label{tab:playlist}
\end{table}

\section*{Acknowledgements}
Warm thanks go out to Clark Kimberling, Peter Moses, Arseniy Akopyan, Ethan Cotterill, and Mark Helman for their generous help.

\bibliographystyle{spmpsci} 
\bibliography{elliptic_billiards_v3,authors_rgk_v1}

\appendix
\section{Triangles: Constructions for Centers}
\label{app:trilinears}
Constructions for a few basic Triangle Centers appear in Figure~\ref{fig:constructions}.

\begin{figure}
    \centering
    \includegraphics[width=\textwidth]{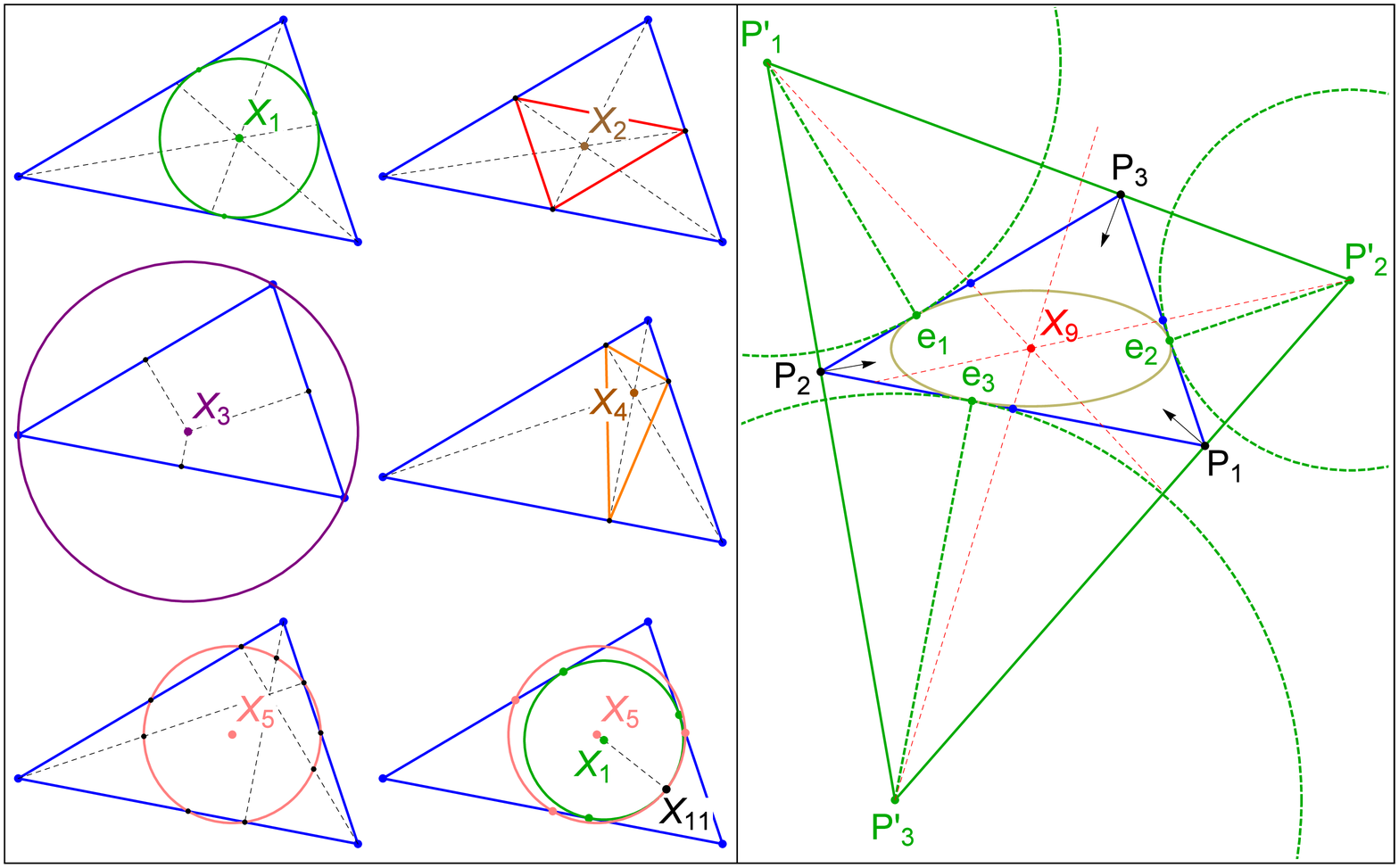}
    \caption{The construction of Basic Triangle Centers $X_i$, as listed in \cite{etc}. \textbf{Left}: The Incenter $X_1$ is the intersection of angular bisectors, and center of the Incircle (green), a circle tangent to the sides at three {\em Intouchpoints} (green dots), its radius is the {\em Inradius} $r$. The Barycenter $X_2$ is where lines drawn from the vertices to opposite sides' midpoints meet. Side midpoints define the {\em Medial Triangle} (red). The Circumcenter $X_3$ is the intersection of perpendicular bisectors, the center of the {\em Circumcircle} (purple) whose radius is the {\em Circumradius} $R$. The Orthocenter $X_4$ is where altitudes concur. Their feet define the {\em Orthic Triangle} (orange). $X_5$ is the center of the 9-Point (or Euler) Circle (pink): it passes through each side's midpoint, altitude feet, and Euler Points \cite{mw}. The Feuerbach Point $X_{11}$ is the single point of contact between the Incircle and the 9-Point Circle. \textbf{Right}: given a reference triangle $P_1P_2P_3$ (blue), the {\em Excenters} $P_1'P_2'P_3'$ are pairwise intersections of lines through the $P_i$ and perpendicular to the bisectors. This triad defines the {\em Excentral Triangle} (green). The {\em Excircles} (dashed green) are centered on the Excenters and are touch each side at an {\em Extouch Point} $e_i,i=1,2,3$. Lines drawn from each Excenter through sides' midpoints (dashed red) concur at the {\em Mittenpunkt} $X_9$. Also shown (brown) is the triangle's {\em Mandart Inellipse}, internally tangent to each side at the $e_i$, and centered on $X_9$. This is identical to the $N=3$ Caustic.}
    \label{fig:constructions}
\end{figure}

Any point on the plane of a triangle $T=P_1P_2P_3$ can be defined by specifying a triple of {\em Trilinear Coordinates} $x:y:z$ which are proportional to the signed distances from $P$ to each side, which makes them invariant under similarity, and reflection transformations.

A {\em Triangle Center} (with respect to a triangle $T={P_1}{P_2}{P_3}$) is defined by Trilinear Coordinates obtained by thrice applying a {\em Triangle Center Function} $h$ to the sidelengths as follows:
 
\begin{equation}
\label{eqn:ftrilins}
    x:y:z {\iff} h(s_1,s_2,s_3):h(s_2,s_3,s_1):h(s_3,s_1,s_2)
\end{equation}

$h$ must (i) be {\em bi-symmetric}, i.e., $h(s_1,s_2,s_3)=h(s_1,s_3,s_2)$, and (ii) homogeneous, $h(t s_1, t s_2, t s_3)=t^n h(s_1,s_2,s_3)$ for some $n$ \cite{kimberling1993_rocky}. Trilinears for nearly 40k Triangle centers are available in \cite{etc}. Trilinears can be converted to Cartesians using \cite{mw}:
  
\begin{equation}
\label{eqn:trilin-cartesian}
X_i|_{\text{cartesian}}=\frac{s_1 x P_1 + s_2 y P_2 + s_3 z P_3}{D}
\end{equation}

Where and $D={s_1}x+{s_2}y+{s_3}z$.

\section{Obtuse Triangles and their Orthics}
\label{app:orthic-incenter}
Let $T=ABC$ be any triangle and $T'=A'B'C'$ its Orthic, Figure~\ref{fig:orthic-incenter}. Let $X'_1$ be the orthic's Incenter. Referring to Figure~\ref{fig:orthic-incenter}.

\begin{lemma}
If $T$ is acute, the Incenter $X'_1$ of $T'$ is the Orthocenter $X_4$ of $T$.
\label{lem:fagnano}
\end{lemma}

\begin{proof}
This is Fagnano's Problem, i.e., the Orthic is the inscribed triangle of minimum perimeter, and the altitudes of $T$ are its bisectors \cite[Section 3.3]{rozikov2018-billiards}. Since the altitudes or $T$ are bisectors of $T'$ this completes the proof.
\end{proof}

\begin{lemma}
If $T$ is obtuse, the Incenter $X'_1$ of $T'$ is the vertex of $T$ subtending the obtuse angle.
\label{lem:pinned}
\end{lemma}

\begin{proof}
Let $B$ be the obtuse angle. Then $B'$ will be on the longest side of $T$ whereas $A'$ (resp. $C'$) will lie on extensions of $BC$ (resp. $AB$), i.e., $A'$ and $C'$ are exterior to $T$. Therefore, $X_4$ will be where altitudes $AA'$ and $CC'$ meet, also exterior to $T$. Since $AA'{\perp}CB$ and $CC'{\perp}AB$, then $CA'$ and $AC'$ are altitudes of triangle $T_e=A{X_4}C$. Since these meet at $B$, the latter is the Orthocenter of $T_e$ and $T'=A'B'C'$ is its Orthic\footnote{The pre-image of $T'$ comprises both $T$ and $T_e$.}. By Lemma~\ref{lem:fagnano}, lines $CA',AC',{X_4}B'$ are bisectors of $T'$, therefore their meetpoint $B$ is the Incenter of $T'$. 
\end{proof}

\begin{corollary}
When $T$ is obtuse, $T_e = A{X_4}C$ is acute and the Excentral Triangle of $T'$.
\end{corollary}

Since all vertices of $T'$ lie on the sides of $T_e$, this is the situation of Lemma~\ref{lem:fagnano}, i.e., $T_e$ is acute. Notice the sides of $T_e$ graze each vertex of $T'$ perpendicular to the bisectors, which is the construction of the Excentral Triangle.

Let $Q$ be a generic triangle and $Q_e$ its Excentral \cite{mw}. Let $\theta_i$ be angles of $Q$ and $\phi_i$ those of $Q_e$ opposite to the $\theta_i$'s. By inspection, $\phi_i=\frac{\pi-\theta_i}{2}$, i.e., all excentral angles are less than $\pi/2$.

\begin{corollary}
If $T$ is obtuse, $X_4$ is an Excenter of $T'$ \cite{coxeter67}.
\end{corollary}

$X_4$ is the intermediate vertex of $T_e$.

%\section{Pythagorean 3-Periodics}
%\label{app:pythagorean}
%\input{appendices/104_app_pythagorean}

\end{document}